\documentclass[10pt]{amsart}

\usepackage[english]{babel}

\usepackage[letterpaper,top=2cm,bottom=2cm,left=3cm,right=3cm,marginparwidth=1.75cm]{geometry}

\usepackage{amsmath}
\usepackage{graphicx}
\usepackage{hyperref}
\usepackage{amsmath,arydshln,multirow}
\usepackage[alphabetic]{amsrefs}
\usepackage{arydshln}
\usepackage{cases}
\usepackage{amsmath}
\usepackage{amsfonts}
\usepackage{bm}
\usepackage{arydshln}
\usepackage{amsfonts,amsmath,amssymb,amscd,bbm,amsthm,mathrsfs,dsfont}
\usepackage{mathrsfs}
\usepackage{pb-diagram}
\usepackage{amssymb}
\usepackage{xypic}
\usepackage{tikz}

\usepackage[color,matrix,arrow]{xy}
\newtheorem{theorem}{Theorem}[section]

\newtheorem{proposition}[theorem]{Proposition}
\newtheorem{lemma}[theorem]{Lemma}

\theoremstyle{definition}

\newtheorem{example}[theorem]{Example}
\newtheorem{proposition-definition}[theorem]{Proposition-Definition}
\newtheorem{corollary}[theorem]{Corollary}

\theoremstyle{remark}
\newtheorem{remark}[theorem]{Remark}

\numberwithin{equation}{section}

\begin{document}

\title{On exchange matrices from string diagrams}


\author{Peigen Cao}
\address{Einstein Institute of Mathematics, Edmond J. Safra Campus, The Hebrew University of Jerusalem, Jerusalem
91904, Israel}
\email{peigencao@126.com}



\date{}

\dedicatory{}


\begin{abstract}Inspired by Fock-Goncharov's amalgamation procedure \cite{Fock-Goncharov-2006}, Shen-Weng introduced string diagrams in \cite{Shen-Weng-2021}, which are very useful to describe many interesting skew-symmetrizable matrices closely related with Lie theory. In this paper, we prove that the skew-symmetrizable matrices from string diagrams are in the smallest class $\mathcal P^\prime$ of skew-symmetrizable matrices containing the $1\times 1$ zero matrix and  closed under mutations and source-sink extensions. This result applies to the exchange matrices of  cluster algebras from double Bruhat cells, unipotent cells, double Bott-Samelson cells and so on.

Our main result can be used to explain why many skew-symmetrizable matrices from Lie theory have reddening sequences. It can be also used to prove some interesting results regarding non-degenerate potentials on many quivers from Lie theory.
\end{abstract}

\maketitle


\tableofcontents

\section{Introduction}
 Cluster algebras $\mathcal A$ were invented by Fomin and Zelevinsky \cite{fz_2002}  as a combinatorial approach to the dual canonical bases of quantized enveloping
algebras \cites{lusztig_1990,lusztig_1991,kashiwara_1990}. Such algebras often arise as the coordinate rings of spaces, such as double Bruhat cells \cite{bfz_2005},  unipotent cells \cite{gls_2011}, double Bott-Samelson cells \cite{Shen-Weng-2021} and arise as the Grothendieck rings of certain monoidal subcategories of the representations of quantum affine algebras \cites{HL_2010,HL_2013}.

An important input to define a cluster algebra $\mathcal A$ is a skew-symmetrizable matrix $B$, called an {\em exchange matrix} of $\mathcal A$. In this paper, we show that the exchange matrices of many interesting cluster algebras from Lie theory are in the smallest class $\mathcal P'$ of skew-symmetrizable matrices containing the $1\times 1$ zero matrix and  closed under mutations and source-sink extensions. This can be used to explain why many skew-symmetrizable matrices from Lie theory have reddening sequences in the sense of \cite{keller_demonet_2020}. It can be also used to prove some interesting results regarding non-degenerate potentials \cite{DWZ_2008} on many quivers from Lie theory.

We remark that source-sink extensions play an important role in Muller's study  \cite{muller-2013} of locally acyclic cluster algebras and in Fei-Weyman's study \cite{Fei-Weyman-2017} of cluster models of upper cluster algebras.

Now we are going to clarify the set-up to introduce our main result in this paper. Let ${\bf i}=(i_1,\cdots,i_{\ell({\bf i})})$ be a sequence with $1\leq i_j\leq n$, where $n$ is a fixed positive integer. Such a sequence is called a {\em $[1,n]$-sequence} with length $\ell({\bf i})$.

Let $A=(a_{ij})_{n\times n}$ be a symmetrizable generalized Cartan matrix and $({\bf i}, {{\bf j}})$ a pair of $[1,n]$-sequences. Then we have a trapezoid $S_{\bf j}^{\bf i}(A)$ and we can consider the triangulations of $S_{\bf j}^{\bf i}(A)$ (see Section \ref{sec3} for details). We remark that one can just identify a triangulation of $S_{\bf j}^{\bf i}(A)$ with a shuffle of the pair $({\bf i}, {{\bf j}})$ of $[1,n]$-sequences.

For each triangulation $T$ of $S_{\bf j}^{\bf i}(A)$, we can define a string diagram $D(S_{\bf j}^{\bf i}(A),T)$. Then we can use the string diagram $D(S_{\bf j}^{\bf i}(A),T)$ to define a skew-symmetrizable matrix $B$, which is called the {\em exchange matrix} of $T$. The matrix $B$ is closely related with the matrices (or quivers, in skew-symmetric case) appearing in \cites{bfz_2005,buan_iyama_reiten_scott_2009,gls_2011,Shen-Weng-2021}.  More precisely:
\begin{itemize}
    \item If $A$ is a Cartan matrix of finite type and the triangulation $T$ (identified with a shuffle of $({\bf i}, {\bf j})$) of  $S_{\bf j}^{\bf i}(A)$ corresponds a reduced word for a pair $(u,v)$ of elements of the Weyl group of the Cartan matrix $A$, then the matrix $B$ here is exactly the same with  the principal part (i.e., exchangeable part) of the extended exchange matrix constructed in \cite{bfz_2005} (see Example \ref{ex:bfz}).  
    \item If $A$ is symmetric, then the matrix $B$ is skew-symmetric and thus corresponds to a quiver $Q_B$.
     In the case of ${\bf i}=\emptyset$ and ${\bf j}$ corresponding to a reduced word of an element of the Weyl group of the Cartan matrix $A$, the quiver $Q_B$ corresponds to the principal part of the ice quiver constructed in \cites{buan_iyama_reiten_scott_2009,gls_2011} (see Example \ref{ex:gls}). 
    \item Generally, the matrix $-B^{\rm T}$ corresponds to the principal part of the matrix constructed in \cite{Shen-Weng-2021}*{Section 3}, because we slightly changed the construction of the exchange matrix of a string diagram in \cite{Shen-Weng-2021} so that it fits well with the matrices or quivers appearing in \cites{bfz_2005,gls_2011}.

\end{itemize}

The following is our main result in this paper.

\begin{theorem}[Theorem \ref{thm:string}]
\label{maithm}
Let $T$ be a triangulation of $S_{\bf j}^{\bf i}(A)$ and $B$ the exchange matrix of $T$. Then $B$ is in the smallest class $\mathcal P'$ of skew-symmetrizable matrices containing the $1\times 1$ zero matrix and  closed under mutations and source-sink extensions.
\end{theorem}

As applications, we have the following corollaries. We refer to \cites{keller_demonet_2020,bm_2020} for the notation of reddening sequences and  \cite{DWZ_2008} for the basic notions on quivers with potentials $(Q,W)$ and Jacobian algebras $J(Q,W)$.

\begin{corollary}[Corollary \ref{app1}, \cite{Shen-Weng-2021}*{Section 4}]
Let $B$ be the exchange matrix of a triangulation $T$ of $S_{\bf j}^{\bf i}(A)$. Then $B$ has a reddening sequence.
\end{corollary}

\begin{corollary}[Corollary \ref{app2}]
Let $B$ be the exchange matrix of a triangulation $T$ of $S_{\bf j}^{\bf i}(A)$.
Suppose that $A$ is symmetric. In this case, $B$ is skew-symmetric and thus corresponds to a quiver $Q_B$. Then $Q_B$ has a unique non-degenerate potential $W_{B}$ (up to right equivalence) which is rigid and its Jacobian algebra $J(Q_{B}, W_{B})$ is finite-dimensional.
\end{corollary}

 The paper is organized as follows. In Section \ref{sec2} we recall the definitions of mutations of matrices and quivers. Then we give the definitions of source-sink extensions and the class $\mathcal P'$. In Section \ref{sec3} we recall the string diagrams introduced by Shen-Weng \cite{Shen-Weng-2021} and explain how to construct a skew-symmetrizable matrix from a string diagram. In Section \ref{sec4} we give the proofs of Theorem \ref{maithm} and its corollaries.

\vspace{2mm}
{\bf Acknowledgement.} I am very grateful to Prof. Kazhdan for providing me with a very comfortable working environment and to Linhui Shen for his helpful comments and interesting discussions. I would also like to thank Jiarui Fei for answering my questions and Fan Qin for discussions on cluster algebras. This project is supported by the ERC (Grant No. 669655) and the NSF of China (Grant No. 12071422).

\section{Preliminaries}\label{sec2}

\subsection{Mutations of matrices and quivers}
Given a set of vertices $I$ and let $B=(b_{ij})_{i,j\in{I}}$ be an integer matrix. We say that $B$ is {\em skew-symmetrizable} if there is an integer diagonal matrix $S={\rm Diag}(s_i)_{i\in I}$ with $s_i>0$ such that $SB$ is skew-symmetric.

Let $B=(b_{ij})_{i,j\in{I}}$ be a skew-symmetrizable integer matrix.
The {\em mutation} of $B$ in direction $k\in I$ is defined to be the new integer matrix $\mu_k(B)= B^\prime=(b_{ij}^\prime)_{i,j\in I}$ given by
$$b_{ij}^\prime=\begin{cases}-b_{ij},&i=k \text{ or } j=k;\\b_{ij}+{\rm sgn}(b_{ik}){\rm max}\{b_{ik}b_{kj},0\},&\text{otherwise.}
\end{cases}$$

It can be checked that $\mu_k(B)$ is still skew-symmetrizable and $\mu_k\mu_k(B)=B$.

If $B=(b_{ij})_{i,j\in I}$ is skew-symmetric, we can represent it using a quiver $Q_{B}$. The vertex set of $Q_{B}$ is $I$. The entry $b_{ij}>0$ if and only if there are $b_{ij}$ many arrows from $j$ to $i$ in $Q_{B}$. We call $Q_{B}$ the {\em usual quiver} of $B$ to distinguish the terminology ``coloured quiver" defined in Section \ref{sec3}. Notice that we can recover $B$ from $Q_B$, because 
$$b_{ij}=|j\rightarrow i|-|i\rightarrow j|.$$

A quiver $Q$ is called a {\em cluster quiver}, if it has no cycles of length $\leq 2$. Clearly, the usual quiver $Q_B$ of a skew-symmetric matrix $B$ is a cluster quiver.

Let $Q$ be a cluster quiver with vertex set $I$. The {\em mutation} of $Q$ at vertex $k\in I$ is the new quiver $\mu_k(Q)$ obtained from $Q$ by the following steps.
\begin{itemize}
    \item [(i)] For each subquiver $i\overset{a}{\longrightarrow} k\overset{b}{\longrightarrow} j$, add a new arrow $i\overset{[ba]}{\longrightarrow} j$;
    \item[(ii)] Reverse all arrows incident with $k$;
    \item[(iii)] Remove the arrows in a maximal set of pairwise disjoint
2-cycles (e.g. $
\xymatrix{\bullet\ar@<.7ex>[r]\ar@<-.7ex>[r]&  \bullet\ar[l]
}
$ yields  $
\xymatrix{\bullet\ar[r]&  \bullet
}$, `$2$-reduction'.)
\end{itemize}

Notice that the resulting quiver $\mu_k(Q)$ is  still a cluster quiver and we  have $\mu_k(\mu_k(Q))=Q$.

\begin{example}The following picture shows how the first quiver changes under the above three steps when we do mutation at $1$.

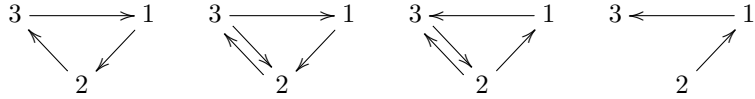
\[\xymatrixrowsep{5mm}
\xymatrixcolsep{5mm}
\xymatrix{3\ar[rr]& &1\ar[ld] &  3\ar[rr]\ar@<.5ex>[rd]& &1\ar[ld] &  3\ar@<.5ex>[rd]& &1\ar[ll] &  3& &1\ar[ll]  
\\
&2\ar[lu]& & & 2\ar@<.5ex>[lu]& & &2\ar@<.5ex>[lu]\ar[ru]& & & 2\ar[ru]&}
\]
\end{example}

Let $B$ be a skew-symmetric matrix and $Q_B$ its usual quiver. The mutation $\mu_k$ of $B$ is compatible with that of $Q_B$, that is, we have
 $$\mu_k(Q_B)=Q_{\mu_k(B)}.$$

\subsection{Source-sink extensions and the class \texorpdfstring{$\mathcal P'$}{Lg}}

Let $B=(b_{ij})_{i,j\in I}$ and $B^\prime=(b_{ij}^\prime)_{i,j\in I^\prime}$ be two skew-symmetrizable matrices. Put $J=I\sqcup I^\prime$ and let $A=(a_{ij})$ be a $J\times J$ skew-symmetrizable matrix. We say that $A$ is a {\em triangular extension} of $B$ and $B^\prime$ if the following three conditions hold.
\begin{itemize}
    \item  $a_{ij}=b_{ij}$ for any $i,j\in I$;
    \item $a_{ij}=b_{ij}'$ for any $i,j\in I'$;
    \item Either $a_{ij}\geq 0$ holds for any $i\in I, j\in I^\prime$ or  $a_{ij}\leq 0$ holds for any $i\in I, j\in I^\prime$.
\end{itemize}

Let $A=(a_{ij})$ be a triangular extension of $B=(b_{ij})_{i,j\in I}$ and $B'=(b_{ij}')_{i,j\in I'}$. If $I'$ has only one element, that is, $B'$ is the $1\times 1$ zero matrix, then we say that $A$ is a {\em source-sink extension} of $B$.

Notice that the notation of triangulation extensions and source-sink extensions make sense for quivers if we restrict to skew-symmetric case.
\begin{example}
Let $Q_B=\xymatrix{1\ar[r]&2}$ and $Q_{B^\prime}=\xymatrix{3\ar[r]&4}$. Then the following quiver is a triangular extension of $Q_B$ and $Q_{B^\prime}$
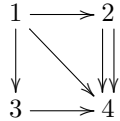
\[
\xymatrix{1\ar[r]\ar[rd]\ar[d]&2\ar@<.5ex>[d]\ar@<-.5ex>[d]\\
3\ar[r]&4}
\]
The following quiver is a soure-sink extension of $Q_B$ by the single point $5$.
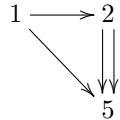
\[
\xymatrix{1\ar[r]\ar[rd]&2\ar@<.5ex>[d]\ar@<-.5ex>[d]\\
&5}
\]
\end{example}

Let $\mathcal P$ be the class of  skew-symmetrizable matrices defined by the following properties.
\begin{itemize}
    \item  The $1\times 1$ zero matrix belongs to $\mathcal P$;
    \item The class $\mathcal P$ is closed under mutations;
    \item If $B, B^\prime\in\mathcal P$ and $A$ is a triangular extension of $B$ and $B^\prime$, then $A\in\mathcal P$. In other words, the class $\mathcal P$ is closed under triangular extensions.
\end{itemize}

Let $\mathcal P'$ be the class of  skew-symmetrizable matrices defined by the following properties.
\begin{itemize}
    \item [(a)] The $1\times 1$ zero matrix belongs to $\mathcal P'$;
    \item[(b)] The class $\mathcal P'$ is closed under mutations;
    \item[(c)] The class $\mathcal P'$ is closed under source-sink extensions.
\end{itemize}

\begin{remark}(i) The classes $\mathcal P$ and $\mathcal P^\prime$ were studied in \cites{lad-2013,bm_2020}. Clearly, we have $\mathcal P'\subseteq \mathcal P$. 

(ii) The skew-symmetrizable matrices in $\mathcal P$ enjoy many nice properties inherited from the $1\times 1$ zero matrix, for example, the existence of reddening sequences in the sense of \cite{keller_demonet_2020}, cf. \cite{bm_2020}*{Theorem 3.3}, and the uniqueness of non-degenerate potential (resp. rigid potential) up to right equivalence \cite{DWZ_2008} in the quiver case, cf. \cite{lad-2013}*{Theorem 4.6}.
\end{remark}

\section{Triangulations, string diagrams and their exchange matrices}\label{sec3}
Inspired by Fock-Goncharov's amalgamation procedure \cite{Fock-Goncharov-2006}, Shen-Weng introduced string diagrams in \cite{Shen-Weng-2021}, which are very useful to describe many interesting skew-symmetrizable matrices closely related with Lie theory.
\subsection{Exchange matrices from string diagrams}

Recall that a {\em symmetrizable generalized Cartan matrix} $A=(a_{ij})_{n\times n}$ is an integer matrix satisfying 
\begin{itemize}
    \item $a_{ii}=2$ for any $i=1,\cdots,n$;
    \item $a_{ij}\leq 0$ for any $i\neq j$;
    \item $a_{ij}=0$ if and only if $a_{ji}=0$;
    \item there is an integer diagonal matrix $S={\rm Diag}(s_1,\cdots,s_n)$ with $s_i>0$ such that $SA$ is a symmetric matrix. 
\end{itemize}

From now on, by a Cartan matrix we always mean a symmetrizable generalized Cartan matrix.

Let ${\bf i}=(i_1,\cdots,i_{\ell({\bf i})})$ be a sequence with $1\leq i_j\leq n$. Such a sequence is called a {\em $[1,n]$-sequence} with length $\ell({\bf i})$.

Given a pair $({\bf i},{\bf j})$ of $[1,n]$-sequences, we can draw a trapezoid $S_{\bf j}^{\bf i}:=S_{\bf j}^{\bf i}(A)$ with bases of lengths $\ell({\bf i})$ and $\ell({\bf j})$. We label the $i$-th unit interval of the top base using the $i$-th integer in ${\bf i}$ and label the $j$-th unit interval of the bottom base using the $j$-th integer in ${\bf j}$.

A {\em triangulation} $T$ of $S_{\bf j}^{\bf i}(A)$ a collection of line segments called {\em diagonals}, each of which connects a marked point on the top to a marked point on the bottom and they divide the trapezoid $S_{\bf j}^{\bf i}(A)$ into triangles.

\begin{remark}(i) The Cartan matrix $A$ plays no role in $S_{\bf j}^{\bf i}=S_{\bf j}^{\bf i}(A)$ and its triangulations $T$, but it will be used when we try to define a skew-symmetrizable matrix from $(S_{\bf j}^{\bf i}(A), T)$.

(ii) We can identify  a triangulation $T$ of $S_{\bf j}^{\bf i}=S_{\bf j}^{\bf i}(A)$ with a shuffle of the pair $({\bf i},{\bf j})$ of $[1,n]$-sequences, see the example followed.
\end{remark}

\begin{example}\label{ex:tri}
Below is a triangulation $T$ of the trapezoid $S_{\bf j}^{\bf i}=S_{\bf j}^{\bf i}(A)$, where $({\bf i},{\bf j})=((1,2,3), (3,1,1,3,2))$.

$$
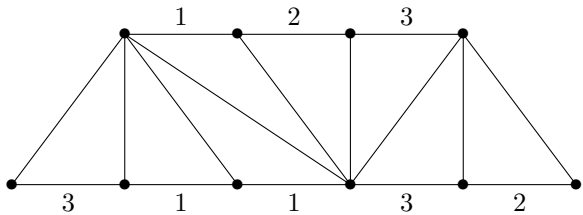
\begin{tikzpicture}[scale=0.50]
  \node (A1) at (3,0) {$\bullet$};
    \node (A2) at (6,0) {$\bullet$};
       \node (A3) at (9,0) {$\bullet$};
          \node (A4) at (12,0) {$\bullet$};

\draw (A1.center) --node [above] {$1$} (A2.center);
\draw (A2.center) --node [above] {$2$} (A3.center);
\draw(A3.center) --node [above] {$3$} (A4.center);

\node (B1) at (0,-4) {$\bullet$};
    \node (B2) at (3,-4) {$\bullet$};
       \node (B3) at (6,-4) {$\bullet$};
          \node (B4) at (9,-4) {$\bullet$};
              \node (B5) at (12,-4) {$\bullet$};
                  \node (B6) at (15,-4) {$\bullet$};
                  
                  \draw (B1.center) --node [below] {$3$} (B2.center);
                  \draw (B2.center) --node [below] {$1$} (B3.center);
                  \draw (B3.center) --node [below] {$1$} (B4.center);
                  \draw (B4.center) --node [below] {$3$} (B5.center);
                  \draw (B5.center) --node [below] {$2$} (B6.center);
                  
                  \draw (A1.center) -- (B1.center);
                  \draw (A1.center) -- (B2.center);
                  \draw (A1.center) -- (B3.center);
                  \draw (A1.center) -- (B4.center);
                  \draw (A2.center) -- (B4.center);
                  \draw (A3.center) -- (B4.center);
                  \draw (A4.center) -- (B4.center);
                  \draw (A4.center) -- (B5.center);
                  \draw (A4.center) -- (B6.center);

\end{tikzpicture}
$$

The triangulation above is identified with the shuffle $(3,1,1,1,2,3,3,2)$ of $({\bf i},{\bf j})=((1,2,3), (3,1,1,3,2))$.
\end{example}

Let $T$ be a triangulation of  $S_{\bf j}^{\bf i}(A)$, we can define a {\em string diagram} $D(T)=D(S_{\bf j}^{\bf i}(A), T)$, which is given as follows:
\begin{itemize}
    \item Draw $n$  horizontal lines across the trapezoid $S_{\bf j}^{\bf i}(A)$; call the $k$-th line from the top the {\em $k$-th level}.
    \item For each triangle whose bottom edge labeled by $k\in\{1,\cdots,n\}$, we put a node labeled by $k$ on the $k$-th level within the triangle.
\item For each triangle whose top edge labeled by $k\in\{1,\cdots,n\}$, we put a node labeled by $-k$ on the $k$-th level within the triangle.
\item The nodes cut the horizontal lines into line segments called {\em string}.  Strings with nodes at both ends
are called {\em closed strings}. The remaining strings are called {\em open strings}.
\end{itemize}

We remark that the Cartan matrix $A$ still plays no role in the definition of string diagrams.

\begin{example}\label{ex:string}
Blue part of Figure \ref{fig:1} shows the string diagram $D(S_{\bf j}^{\bf i}(A), T)$ associated to $(S_{\bf j}^{\bf i}(A), T)$ in Example \ref{ex:tri}.

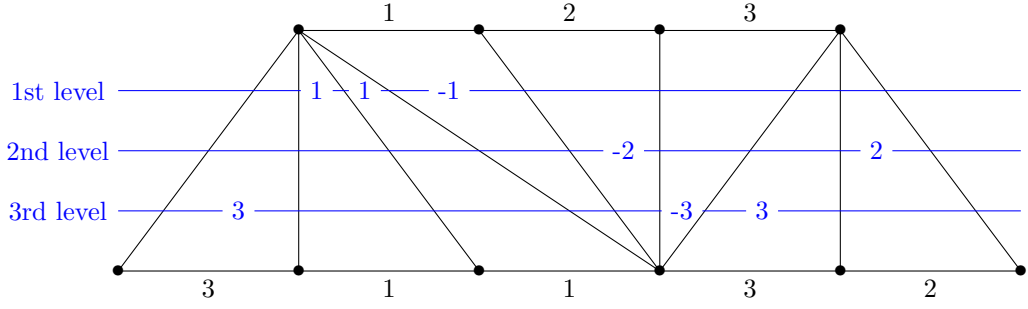
\begin{figure}
  \centering
\begin{tikzpicture}[scale=0.80]
  \node (A1) at (3,0) {$\bullet$};
    \node (A2) at (6,0) {$\bullet$};
       \node (A3) at (9,0) {$\bullet$};
          \node (A4) at (12,0) {$\bullet$};

\draw (A1.center) --node [above] {$1$} (A2.center);
\draw (A2.center) --node [above] {$2$} (A3.center);
\draw(A3.center) --node [above] {$3$} (A4.center);

\node (B1) at (0,-4) {$\bullet$};
    \node (B2) at (3,-4) {$\bullet$};
       \node (B3) at (6,-4) {$\bullet$};
          \node (B4) at (9,-4) {$\bullet$};
              \node (B5) at (12,-4) {$\bullet$};
                  \node (B6) at (15,-4) {$\bullet$};
                  
                  \draw (B1.center) --node [below] {$3$} (B2.center);
                  \draw (B2.center) --node [below] {$1$} (B3.center);
                  \draw (B3.center) --node [below] {$1$} (B4.center);
                  \draw (B4.center) --node [below] {$3$} (B5.center);
                  \draw (B5.center) --node [below] {$2$} (B6.center);
                  
                  \draw (A1.center) -- (B1.center);
                  \draw (A1.center) -- (B2.center);
                  \draw (A1.center) -- (B3.center);
                  \draw (A1.center) -- (B4.center);
                  \draw (A2.center) -- (B4.center);
                  \draw (A3.center) -- (B4.center);
                  \draw (A4.center) -- (B4.center);
                  \draw (A4.center) -- (B5.center);
                  \draw (A4.center) -- (B6.center);
                  
              \node (l1) at (3.3,-1) {\textcolor{blue}{1}};
                   \node (l2) at (4.1,-1) {\textcolor{blue}{1}};
                   \node (l3) at (5.5,-1) {\textcolor{blue}{-1}};
                   \draw[blue] (0,-1)--(l1)--(l2)--(l3)--(15,-1);
                   
                   \node (m1) at (8.4,-2) {\textcolor{blue}{-2}};
                   \node (m2) at (12.6,-2) {\textcolor{blue}{2}};
                   \draw[blue] (0,-2)--(m1)--(m2)--(15,-2);
                   \node (n1) at (2,-3) {\textcolor{blue}{3}};
                   \node (n2) at (9.37,-3) {\textcolor{blue}{-3}};
                   \node (n3) at (10.7,-3) {\textcolor{blue}{3}};
                   \draw[blue] (0,-3)--(n1)--(n2)--(n3)--(15,-3);
                   \node at (-1,-1) {\textcolor{blue}{1st level}};
                    \node at (-1,-2) {\textcolor{blue}{2nd level}};
                     \node at (-1,-3) {\textcolor{blue}{3rd level}};
\end{tikzpicture}
  \caption{An example of string diagram.}\label{fig:1}
\end{figure}

\end{example}

In a string diagram, we always label the closed strings on the $k$-th level from left to right by $\binom{k}{1},\binom{k}{2},\binom{k}{3},\cdots.$
Denote by $I$ the set of closed strings in a string diagram. We always view $I$ as an ordered set with the order: ${k_1\choose i}\preceq{k_2\choose j}$ if $k_1<k_2$ or $k_1=k_2$ and $i\leq j$. For example, in Figure \ref{fig:1} $$I=\{{1 \choose 1}\prec{1 \choose 2}\prec{2 \choose 1}\prec{3 \choose 1}\prec{3 \choose 2}\}.$$ 
This order will be used when we try to write down a matrix indexed by $I$.

To every string diagram $D(S_{\bf j}^{\bf i}(A), T)$ (and hence to every pair $(S_{\bf j}^{\bf i}(A), T)$), we construct a skew-symmetrizable matrix $B=(b_{ij})_{i,j\in I}$, where $I$ is the set of closed strings of the string diagram.  The entries of $B$ are determined by $$B=\sum_{\text{nodes}\; m}B^{(m)},$$
where each $B^{(m)}=(b_{ij}^{(m)})$ is defined as follows:
Suppose that the node $m$ belongs to the $k$-th level and $x$ and $y$ are the strings in $k$-th level with the node $m$ an endpoint. Namely, we have $$\xymatrix{\ar@{-}[r]^x &m\ar@{-}[r]^y&}$$ 
Denote by 
$$\epsilon(m)=\begin{cases}1,&\text{if}\;m\;\text{is labeled by}\;k;\\
-1,&\text{if}\;m\;\text{is labeled by}\;-k.\end{cases}$$
\begin{itemize}
    \item If both $x$ and $y$ are closed strings, then we define $$b_{xy}^{(m)}=-b_{yx}^{(m)}=1\cdot\epsilon(m).$$
    \item Let $z$ be a closed string on the $j$-th level with $j\neq k$ such that $z$ intersects with the triangle containing the node $m$. If $x$ is a closed string, we set 
    $$b_{xz}^{(m)}=\frac{a_{kj}}{2}\cdot \epsilon(m) \;\;\;\text{and}\;\;\;b_{zx}^{(m)}=-\frac{a_{jk}}{2}\cdot \epsilon(m).$$
If  $y$ is a closed string, we set 
    $$b_{yz}^{(m)}=-\frac{a_{kj}}{2}\cdot \epsilon(m) \;\;\;\text{and}\;\;\;b_{zy}^{(m)}=\frac{a_{jk}}{2}\cdot \epsilon(m).$$    
\item The remain entries of $B^{(m)}$ are zero.
\end{itemize}

The matrix $B$ constructed above is called the {\em exchange matrix} of  $T$. 

Now we explain why $B$ is skew-symmetrizable (see also \cite{Shen-Weng-2021}*{Section 3.1}).
Let $S={\rm Diag}(s_1,\cdots,s_n)$ be an integer diagonal matrix with $s_i>0$ such that  $SA$ is symmetric. If $x$ is a closed string of $D(S_{\bf j}^{\bf i}(A), T)$ on the $k$-th level, we define $s_x^\prime:=s_k$. Then $S^\prime={\rm Diag}(s_x^\prime)_{x\in I}$ is a skew-symmetrizer of $B$, that is, $S^\prime B$ is skew-symmetric. Thus $B$ is skew-symmetrizable. Clearly, if $A$ is symmetric, then $B$ is skew-symmetric.

\begin{remark}
(i) By the construction of $B$, we know that if $x$ with endpoints $m_1,m_2$ and $z$ with endpoints $n_1,n_2$ are two closed strings on different levels,
then the $(x,z)$-entry $b_{xz}$ of $B$ is given by
$$b_{xz}=b_{xz}^{(m_1)}+b_{xz}^{(m_2)}+b_{xz}^{(n_1)}+b_{xz}^{(n_2)}.$$

 (ii) Although $B^{(m)}$ have entries with denominator $2$, when we sum up all nodes $m$, the resulting matrix $B$ is an integer matrix.  In fact, if $x$ and $y$ are on the same level, then $b_{xy}=-b_{yx}\in\{0,\pm 1\}$.  If $x$ is on the level $k$ and $z$ is on the level $j$ with $j\neq k$, then $b_{xz}\in\{0,\pm a_{kj}\}$. 
 
 (iii) We can use a {\em coloured quiver} to represent $B$. Because the absolute values of non-zero entries $b_{xz}$ of $B$ are determined by the levels data of $x$ and $z$, the only extra data we need to record are the signs of $b_{xz}$. We record the signs using two different types of arrows. 
If $x$ and $z$ are on the same level and $b_{xz}=-1$, then we draw an (horizontal) arrow of the form $\xymatrix{x\ar[r]&z}$. If $x$ is on the level $k$ and $z$ is on the level $j\neq k$ and $b_{xz}=a_{kj}<0$, then we draw an (inclined) arrow of the form $\xymatrix{x\ar@{~>}[r]&z}$.

(iv) The matrix $-B^{\rm T}$ corresponds to the principal part of the matrix constructed in \cite{Shen-Weng-2021}*{Section 3}, because we slightly changed the construction of the exchange matrix of a string diagram in \cite{Shen-Weng-2021}.
\end{remark}

\begin{example} Let $D(S_{\bf j}^{\bf i}(A), T)$ be the string diagram in Figure \ref{fig:1}. The set of closed strings is the following ordered set.
$$I=\{{1 \choose 1}\prec{1 \choose 2}\prec{2 \choose 1}\prec{3 \choose 1}\prec{3 \choose 2}\}.$$ 
We will take different Cartan matrices and look at the coloured quivers and the exchange matrices of $D(S_{\bf j}^{\bf i}(A), T)$.

(i) Take the Cartan matrices $A=\begin{bmatrix}2&-5&-7\\
-5&2&-9\\-7&-9&2\end{bmatrix}$ and $A^\prime=\begin{bmatrix}2&-6&-8\\
-3&2&-10\\-4&-10&2\end{bmatrix}$. Then the string diagrams $D(S_{\bf j}^{\bf i}(A), T)$ and $D(S_{\bf j}^{\bf i}(A^\prime), T)$ have the same coloured quiver. Their common coloured quiver is as follows:
$${\xymatrixrowsep{4mm}
\xymatrixcolsep{4mm}
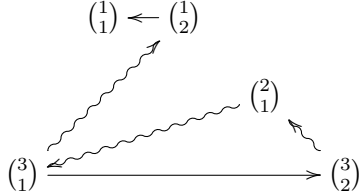
\xymatrix{
&{1 \choose 1}&{1 \choose 2}\ar[l]&&\\
&&&{2 \choose 1}\ar@{~>}[llld]&\\  
{3 \choose 1}\ar@{~>}[rruu]\ar[rrrr]&&&&{3 \choose 2}\ar@{~>}[lu]
}}$$
However, $D(S_{\bf j}^{\bf i}(A), T)$ and $D(S_{\bf j}^{\bf i}(A^\prime), T)$ has different exchange matrices, which are given as follows:

$$B=\begin{bmatrix}0&1&0&0&0\\ -1&0&0&7&0\\ 0&0&0&-9&9\\
0&-7&9&0&-1\\ 0&0&-9&1&0\end{bmatrix}\;\;\;\text{and}\;\;\;
B^\prime=\begin{bmatrix}0&1&0&0&0\\ -1&0&0&8&0\\ 0&0&0&-10&10\\
0&-4&10&0&-1\\ 0&0&-10&1&0\end{bmatrix}.$$

(ii) Take the Cartan matrix $A=\begin{bmatrix}2&-6&-8\\
-3&2&0\\-4&0&2\end{bmatrix}$. Then the coloured quiver and the exchange matrix of $D(S_{\bf j}^{\bf i}(A), T)$ are given as follows:

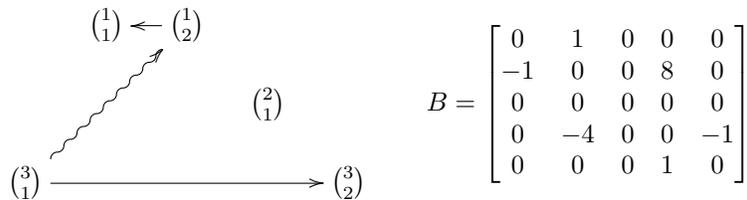
\[\begin{array}{llll}
\begin{array}{llll}
 {\xymatrixrowsep{4mm}
\xymatrixcolsep{4mm}  \xymatrix{
&{1 \choose 1}&{1 \choose 2}\ar[l]&&\\
&&&{2 \choose 1}&\\  
{3 \choose 1}\ar@{~>}[rruu]\ar[rrrr]&&&&{3 \choose 2}
} }\end{array}&
\begin{array}{llll}
B=\begin{bmatrix}0&1&0&0&0\\ -1&0&0&8&0\\ 0&0&0&0&0\\
0&-4&0&0&-1\\ 0&0&0&1&0\end{bmatrix}\end{array}\end{array}\]

\end{example}

\begin{example}\label{ex:bfz}
Take the Cartan matrix $A=\begin{bmatrix}2&-1\\
-1&2\end{bmatrix}$ and consider the following string diagram $D(S_{\bf j}^{\bf i}(A), T)$.
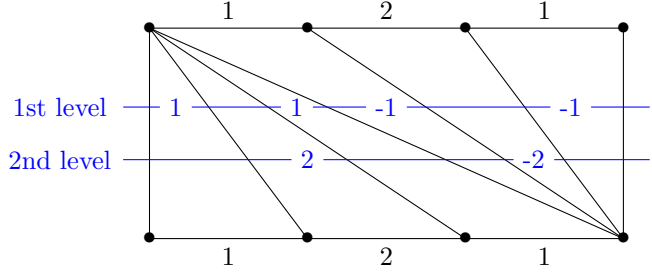
$$\begin{tikzpicture}[scale=0.70]
  \node (A1) at (3,0) {$\bullet$};
    \node (A2) at (6,0) {$\bullet$};
       \node (A3) at (9,0) {$\bullet$};
          \node (A4) at (12,0) {$\bullet$};

\draw (A1.center) --node [above] {$1$} (A2.center);
\draw (A2.center) --node [above] {$2$} (A3.center);
\draw(A3.center) --node [above] {$1$} (A4.center);

\node (B1) at (3,-4) {$\bullet$};
    \node (B2) at (6,-4) {$\bullet$};
       \node (B3) at (9,-4) {$\bullet$};
          \node (B4) at (12,-4) {$\bullet$};

                  \draw (B1.center) --node [below] {$1$} (B2.center);
                  \draw (B2.center) --node [below] {$2$} (B3.center);
                  \draw (B3.center) --node [below] {$1$} (B4.center);

                  \draw (A1.center) -- (B1.center);
                  \draw (A1.center) -- (B2.center);
                  \draw (A1.center) -- (B3.center);
                  \draw (A1.center) -- (B4.center);
                  \draw (A2.center) -- (B4.center);
                  \draw (A3.center) -- (B4.center);
                  \draw (A4.center) -- (B4.center);

              \node (l1) at (3.5,-1.5) {\textcolor{blue}{1}};
                   \node (l2) at (5.8,-1.5) {\textcolor{blue}{1}};
                   \node (l3) at (7.5,-1.5) {\textcolor{blue}{-1}};
                    \node (l4) at (11,-1.5) {\textcolor{blue}{-1}};
                   \draw[blue] (2.5,-1.5)--(l1)--(l2)--(l3)--(l4)--(12.5,-1.5);
                   
                   \node (m1) at (6,-2.5) {\textcolor{blue}{2}};
                   \node (m2) at (10.3,-2.5) {\textcolor{blue}{-2}};
                   \draw[blue] (2.5,-2.5)--(m1)--(m2)--(12.5,-2.5);

                   \node at (1.3,-1.5) {\textcolor{blue}{1st level}};
                    \node at (1.3,-2.5) {\textcolor{blue}{2nd level}};
                    
\end{tikzpicture}
$$

Then the coloured quiver and the exchange matrix of $D(S_{\bf j}^{\bf i}(A), T)$ are given as follows:

\[\begin{array}{llll}
\begin{array}{llll}
 {\xymatrixrowsep{6mm}
\xymatrixcolsep{6mm}  
 \xymatrix{
{1 \choose 1}\ar@{~>}[rd]&&{1 \choose 2}\ar[ll]\ar[r]&{1 \choose 3}\ar@{~>}[lld]\\
&{2 \choose 1}\ar@{~>}[ru]&&
}
 }\end{array}&
\begin{array}{llll}
B=\begin{bmatrix}0&1&0&-1\\-1&0&-1&1\\0&1&0&-1\\1&-1&1&0\end{bmatrix}\end{array}\end{array}\]

Under the new order $\{{1\choose 1}< {2\choose 1}< {1\choose 2}<{1\choose 3}\}$ of the vertex set $I=\{{1\choose 1}\prec{1\choose 2}\prec {1\choose 3}\prec {2\choose 1}\}$, the exchange matrix $B$ becomes 
$$B^\prime=\begin{bmatrix}0&-1&1&0\\1&0&-1&1\\-1&1&0&-1\\0&-1&1&0\end{bmatrix}.$$
One can see that $B^\prime$ is exactly the same with the exchange matrix in \cite{bfz_2005}*{Example 2.18}.
\end{example}

\begin{proposition}\label{pro:sw1}
\cite{Shen-Weng-2021}*{Proposition 3.7} Let $T$ be a triangulation of $S_{\bf j}^{\bf i}(A)$ and $T^\prime$ a triangulation of $S_{\bf j}^{\bf i}(A)$ obtained from $T$ by flip a diagonal inside a quadrilateral whose top edge and bottom edge are labelled by $i$ and $j$ respectively (see as follows).
$$
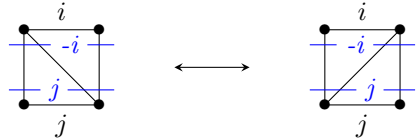
\begin{tikzpicture}
  \node (n1) at (0,0) {$\bullet$};
  \node (n2) at (0,1) {$\bullet$};
  \node (n3) at (1,1) {$\bullet$};
  \node (n4) at (1,0) {$\bullet$};
  \draw (n1.center) -- (n2.center);
\draw (n2.center) --node [above] {$i$} (n3.center);
 \draw (n3.center) -- (n4.center);
\draw (n4.center) --node [below] {$j$} (n1.center);
\draw(n2.center) -- (n4.center);

 \node (v1) at (4,0) {$\bullet$};
  \node (v2) at (4,1) {$\bullet$};
  \node (v3) at (5,1) {$\bullet$};
  \node (v4) at (5,0) {$\bullet$};
  \draw (v1.center) -- (v2.center);
\draw (v2.center) --node [above] {$i$} (v3.center);
 \draw (v3.center) -- (v4.center);
\draw (v4.center) --node [below] {$j$} (v1.center);
\draw(v1.center) -- (v3.center);

\draw[stealth-stealth](2,0.5)--(3,0.5);

\node (l1) at (0.6,0.8) {\textcolor{blue}{-i}};
\node (l2) at (0.4,0.2) {\textcolor{blue}{j}};
\node (l3) at (4.4,0.8) {\textcolor{blue}{-i}};
\node (l4) at (4.6,0.2) {\textcolor{blue}{j}};

\draw[blue](-0.2,0.8)--(l1)--(1.2,0.8);
\draw[blue](-0.2,0.2)--(l2)--(1.2,0.2);
\draw[blue](3.8,0.8)--(l3)--(5.2,0.8);
\draw[blue](3.8,0.2)--(l4)--(5.2,0.2);

\end{tikzpicture}
$$

\begin{itemize}
    \item [(i)] If $i\neq j$, then the exchange matrices of $T$ and $T^\prime$ are the same;
    \item[(ii)] If $i=j$, then the exchange matrix of $T$ and that of $T^\prime$ are obtained from each other by mutating the vertex corresponding to the closed string on the $i$-th level with endpoints $j=i$ and $-i$ inside the quadrilateral.
\end{itemize}
\end{proposition}

\begin{proposition}\label{pro:sw2}
\cite{Shen-Weng-2021}*{Proposition 3.3}
 Let $T, T^\prime$ be two triangulations of $S_{\bf j}^{\bf i}(A)$ and $B, B^\prime$ their exchange matrices. Then 
 \begin{itemize}
     \item [(i)] $T$ and $T^\prime$ can be obtained from each other by a sequence of flips;
     \item[(ii)] $B$ and $B^\prime$ can be obtained from each other by a sequence of mutations.
 \end{itemize}
\end{proposition}

\subsection{Examples for the case \texorpdfstring{$({\bf i},{\bf j})=(\emptyset,{\bf j})$}{Lg}}
In this subsection, we focus on the case $({\bf i},{\bf j})=(\emptyset,{\bf j})$. In this case, the trapezoid $S_{\bf j}^\emptyset(A)$ is in fact a triangle and there is only one available triangulation $T_{\bf j}$ in this case, which is as follows:
$$
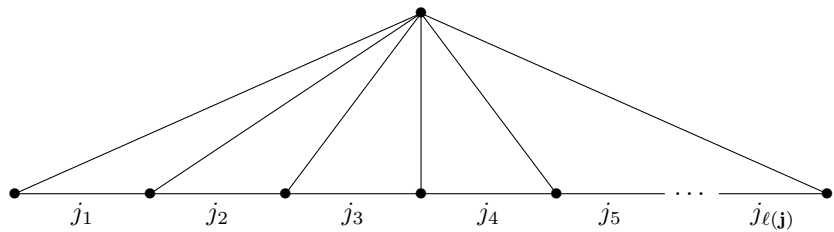
\begin{tikzpicture}[scale=0.60]
  
       \node (A) at (9,0) {$\bullet$};
        
\node (B1) at (0,-4) {$\bullet$};
    \node (B2) at (3,-4) {$\bullet$};
       \node (B3) at (6,-4) {$\bullet$};
          \node (B4) at (9,-4) {$\bullet$};
              \node (B5) at (12,-4) {$\bullet$};
              \node(B6) at (15,-4) {$\dots$};
                  \node (B7) at (18,-4) {$\bullet$};
                  
                  \draw (B1.center) --node [below] {$j_1$} (B2.center);
                  \draw (B2.center) --node [below] {$j_2$} (B3.center);
                  \draw (B3.center) --node [below] {$j_3$} (B4.center);
                  \draw (B4.center) --node [below] {$j_4$} (B5.center);
                  \draw (B5.center) --node [below] {$j_5$} (B6);
                  \draw (B6) --node [below] {$j_{\ell({\bf j})}$} (B7.center);
                  \draw (A.center) -- (B1.center);
                  \draw (A.center) -- (B2.center);
                  \draw (A.center) -- (B3.center);
                  \draw (A.center) -- (B4.center);
                  \draw (A.center) -- (B5.center);
                  \draw (A.center) -- (B7.center);
\end{tikzpicture}
$$

Thus for each $[1,n]$-sequence ${\bf j}$, there is an exchange matrix $B_{\bf j}$ associated with $S_{\bf j}^\emptyset(A)$. If $A$ is symmetric, then $B_{\bf j}$ is skew-symmetric. In this case, we can represent $B_{\bf j}$ using its usual quiver $Q_{\bf j}$. One can see that the quiver $Q_{\bf j}$ corresponds to the principal part of the ice quivers appearing  in \cites{buan_iyama_reiten_scott_2009,gls_2011}.

\begin{example}\label{ex:gls}
Take the Cartan matrix $A=\begin{bmatrix} 2&-3&-2\\-3&2&-2\\-2&-2&2
\end{bmatrix}$ and ${\bf j}=(j_1,\cdots,j_{10})=(1,2,1,3,1,2,1,2,3,2)$.
Then the triangulation $T_{\bf j}$ and the string diagram of $S_{\bf j}^{\emptyset}(A)$ are as follows:
$$
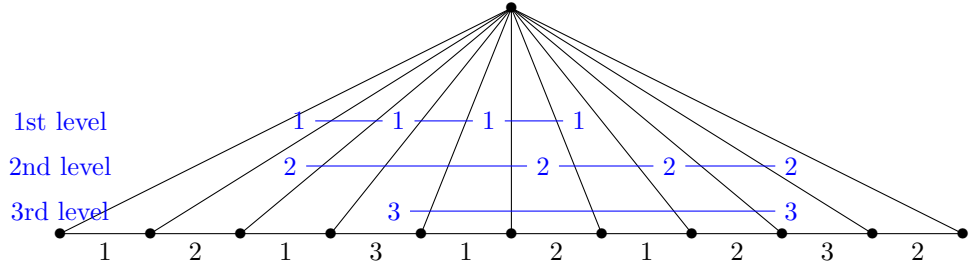
\begin{tikzpicture}[scale=0.60]
  
       \node (a) at (0,0) {$\bullet$};
       \node (b1) at (-10,-5) {$\bullet$};
         \node (b2) at (-8,-5) {$\bullet$};
           \node (b3) at (-6,-5) {$\bullet$};
             \node (b4) at (-4,-5) {$\bullet$};
               \node (b5) at (-2,-5) {$\bullet$};
                 \node (b6) at (0,-5) {$\bullet$};
                  \node (b7) at (2,-5) {$\bullet$};
                  \node (b8) at (4,-5) {$\bullet$};
           \node (b9) at (6,-5) {$\bullet$};
             \node (b10) at (8,-5) {$\bullet$};
               \node (b11) at (10,-5) {$\bullet$};
\draw (b1.center) --node [below] {$1$} (b2.center);    
\draw (b2.center) --node [below] {$2$} (b3.center);  
\draw (b3.center) --node [below] {$1$} (b4.center);    
\draw (b4.center) --node [below] {$3$} (b5.center); 
\draw (b5.center) --node [below] {$1$} (b6.center);
\draw (b6.center) --node [below] {$2$} (b7.center);
\draw (b7.center) --node [below] {$1$} (b8.center);
\draw (b8.center) --node [below] {$2$} (b9.center);
\draw (b9.center) --node [below] {$3$} (b10.center);
\draw (b10.center) --node [below] {$2$} (b11.center);
               \draw (a.center) -- (b1.center);
                 \draw (a.center) -- (b2.center);
                   \draw (a.center) -- (b3.center);
                     \draw (a.center) -- (b4.center);
                       \draw (a.center) -- (b5.center);
                         \draw (a.center) -- (b6.center);
                           \draw (a.center) -- (b7.center);
                             \draw (a.center) -- (b8.center);
                               \draw (a.center) -- (b9.center);
                                 \draw (a.center) -- (b10.center);
                                  \draw (a.center) -- (b11.center);
            
              \node (l1) at (-4.7,-2.5) {\textcolor{blue}{1}};
              \node (l2) at (-2.5,-2.5) {\textcolor{blue}{1}};
                \node (l3) at (-0.5,-2.5) {\textcolor{blue}{1}};
               \node (l4) at (1.5,-2.5) {\textcolor{blue}{1}};
                  \draw[blue] (l1)--(l2)--(l3)--(l4);
                  
                  \node (m1) at (-4.9,-3.5) {\textcolor{blue}{2}};
              \node (m2) at (0.7,-3.5) {\textcolor{blue}{2}};
               \node (m3) at (3.5,-3.5) {\textcolor{blue}{2}};
                \node (m4) at (6.2,-3.5) {\textcolor{blue}{2}};
                \draw[blue] (m1)--(m2)--(m3)--(m4);
                
                \node (n1) at (-2.6,-4.5) {\textcolor{blue}{3}};
              \node (n2) at (6.2,-4.5) {\textcolor{blue}{3}};
              \draw[blue] (n1)--(n2);

                   \node at (-10,-2.5) {\textcolor{blue}{1st level}};
                    \node at (-10,-3.5) {\textcolor{blue}{2nd level}};                 \node at (-10,-4.5) {\textcolor{blue}{3rd level}};
        
\end{tikzpicture}
$$

Because the Cartan matrix $A$ here is symmetric, we know that the exchange matrix $B_{\bf j}$ of the above triangulation is skew-symmetric. The usual quiver 
 $Q_{\bf j}$ of $B_{\bf j}$ is as follows:
 \[\xymatrixrowsep{5mm}
\xymatrixcolsep{5mm}
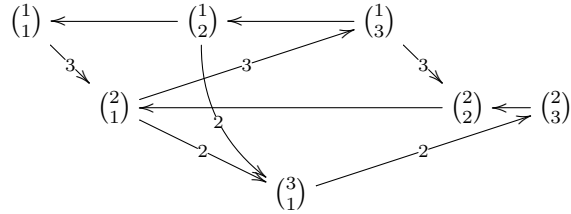
\xymatrix{{1 \choose 1}\ar[rd]|3 &&{1 \choose 2}\ar[ll]\ar@/_1pc/[rdd]|2 & &{1 \choose 3}\ar[ll]\ar[rd]|3&& \\
 &{2 \choose 1}\ar[rrru]|3\ar[rrd]|2&& &&{2 \choose 2}\ar[llll]&{2 \choose 3}\ar[l] \\
  &&&{3 \choose 1}\ar[rrru]|2 &&&}
\]
where $\xymatrix{j\ar[r]|q&i}$ means that there are $q$ arrows from $j$ to $i$. (The readers can compare this example
with the one in \cite{gls_2011}*{Section 13.2, Page 405}.)
\end{example}

\begin{example}
Take the Cartan matrix $A=\begin{bmatrix}2&-6&-8\\
-3&2&-10\\-4&-10&2\end{bmatrix}$ and ${\bf j}=(j_1,\cdots,j_7)=(1,2,1,3,1,3,2)$.  Then the triangulation $T_{\bf j}$ and the string diagram of $S_{\bf j}^{\emptyset}(A)$ are as follows:
$$
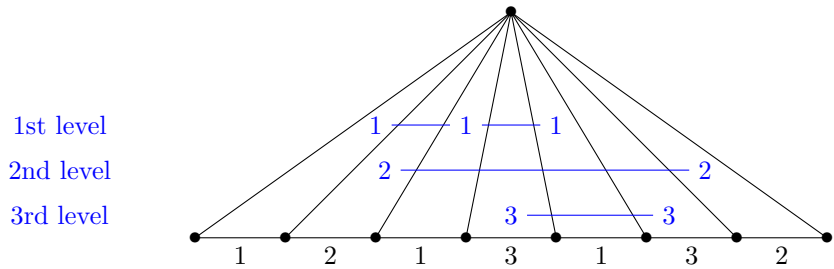
\begin{tikzpicture}[scale=0.60]
  
       \node (a) at (0,0) {$\bullet$};
       \node (b1) at (-7,-5) {$\bullet$};
         \node (b2) at (-5,-5) {$\bullet$};
           \node (b3) at (-3,-5) {$\bullet$};
             \node (b4) at (-1,-5) {$\bullet$};
               \node (b5) at (1,-5) {$\bullet$};
                 \node (b6) at (3,-5) {$\bullet$};
                  \node (b7) at (5,-5) {$\bullet$};
                  \node (b8) at (7,-5) {$\bullet$};
           
\draw (b1.center) --node [below] {$1$} (b2.center);    
\draw (b2.center) --node [below] {$2$} (b3.center);  
\draw (b3.center) --node [below] {$1$} (b4.center);    
\draw (b4.center) --node [below] {$3$} (b5.center); 
\draw (b5.center) --node [below] {$1$} (b6.center);
\draw (b6.center) --node [below] {$3$} (b7.center);
\draw (b7.center) --node [below] {$2$} (b8.center);
               \draw (a.center) -- (b1.center);
                 \draw (a.center) -- (b2.center);
                   \draw (a.center) -- (b3.center);
                     \draw (a.center) -- (b4.center);
                       \draw (a.center) -- (b5.center);
                         \draw (a.center) -- (b6.center);
                           \draw (a.center) -- (b7.center);
                             \draw (a.center) -- (b8.center);

              \node (l1) at (-3,-2.5) {\textcolor{blue}{1}};
              \node (l2) at (-1,-2.5) {\textcolor{blue}{1}};
                \node (l3) at (1,-2.5) {\textcolor{blue}{1}};
               {\textcolor{blue}{1}};
                  \draw[blue] (l1)--(l2)--(l3);
                  
                  \node (m1) at (-2.8,-3.5) {\textcolor{blue}{2}};
              \node (m2) at (4.3,-3.5) {\textcolor{blue}{2}};
            
                \draw[blue] (m1)--(m2);
                
                \node (n1) at (0,-4.5) {\textcolor{blue}{3}};
              \node (n2) at (3.5,-4.5) {\textcolor{blue}{3}};
              \draw[blue] (n1)--(n2);

                   \node at (-10,-2.5) {\textcolor{blue}{1st level}};
                    \node at (-10,-3.5) {\textcolor{blue}{2nd level}};                 \node at (-10,-4.5) {\textcolor{blue}{3rd level}};
        
\end{tikzpicture}
$$

Then the coloured quiver and the exchange matrix of the above triangulation are given as follows:

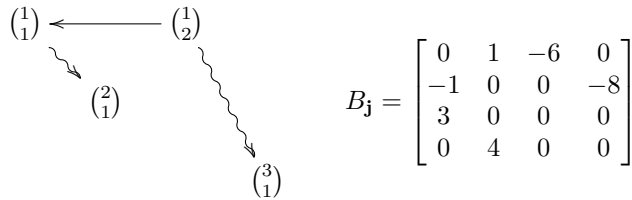
\[\begin{array}{llll}
\begin{array}{llll}
 {\xymatrixrowsep{4mm}
\xymatrixcolsep{4mm}  \xymatrix{
{1 \choose 1}\ar@{~>}[rd]&&{1 \choose 2}\ar[ll]\ar@{~>}[rdd]&\\
&{2 \choose 1}&&\\  
&&&{3 \choose 1}
}  }\end{array}&
\begin{array}{llll}
B_{\bf j}=\begin{bmatrix}0&1&-6&0\\ -1&0&0&-8\\ 3&0&0&0\\0&4&0&0\end{bmatrix}\end{array}\end{array}\]

\end{example}

\section{Exchange matrices from string diagrams are in the class in \texorpdfstring{$\mathcal P'$}{Lg}}
\label{sec4}

\subsection{Reduce \texorpdfstring{$({\bf i},{\bf j})$}{Lg} to \texorpdfstring{$(\emptyset,{\bf i}^{-1}\circ{\bf j})$}{Lg}}

Let $S_{\bf j}^{\bf i}(A)$ be the trapezoid associated with a pair $({\bf i},{\bf j})$ of $[1,n]$-sequences, say 
$${\bf i}=(i_1,\dots,i_{\ell(\bf i)})\;\;\;\text{and}\;\;\;{\bf j}=(j_1,\dots,j_{\ell(\bf j)}).$$ We define

\begin{eqnarray}
L_d(S_{\bf j}^{\bf i}(A))&:=&S_{{\bf j}^\prime}^{{\bf i}^\prime}(A),\;\text{where}\;\; {\bf i}^\prime=(j_1,i_1,\cdots,i_{\ell({\bf i})})\;\;\text{and}\;\;{\bf j}^\prime=(j_2,\dots,j_{\ell({\bf j})});  \nonumber\\
L_u(S_{\bf j}^{\bf i}(A))&:=&S_{{\bf j}^\prime}^{{\bf i}^\prime}(A),\;\text{where}\;\; {\bf i}^\prime=(i_2,\cdots,i_{\ell({\bf i})})\;\;\text{and}\;\;{\bf j}^\prime=(i_1,j_1,\dots,j_{\ell({\bf j})});  \nonumber\\
R_d(S_{\bf j}^{\bf i}(A))&:=& S_{{\bf j}^\prime}^{{\bf i}^\prime}(A),\;\text{where}\;\; {\bf i}^\prime=(i_1,\cdots,i_{\ell({\bf i})},j_{\ell({\bf j})})\;\;\text{and}\;\;{\bf j}^\prime=(j_1,\dots,j_{\ell({\bf j})-1}); \nonumber\\
R_u(S_{\bf j}^{\bf i}(A))&:=&S_{{\bf j}^\prime}^{{\bf i}^\prime}(A),\;\text{where}\;\; {\bf i}^\prime=(i_1,\cdots,i_{\ell({\bf i})-1})\;\;\text{and}\;\;{\bf j}^\prime=(j_1,\dots,j_{\ell({\bf j})},i_{\ell({\bf i})}). \nonumber
\end{eqnarray}

We know that each triangulation $T$ of $S_{\bf j}^{\bf i}(A)$ cuts the trapezoid $S_{\bf j}^{\bf i}(A)$ into $\ell({\bf i})+\ell({\bf j})$ small triangles. Each  triangle of $T$ is labeled by an integer in ${\bf i}$ or ${\bf j}$ and it looks like
$$
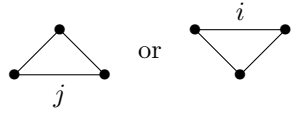
\begin{tikzpicture}[scale=0.60]
\node (v1) at (0,0) {$\bullet$};
\node (v2) at (-1,-1) {$\bullet$};
\node (v3) at (1,-1) {$\bullet$};
\draw (v1.center)--(v2.center)-- node[below]{$j$}(v3.center)--(v1.center);
\node at (2,-0.5) {or};
\node (l1) at (4,-1) {$\bullet$};
\node (l2) at (3,0) {$\bullet$};
\node (l3) at (5,0) {$\bullet$};
\draw (l1.center)--(l2.center)-- node[above]{$i$}(l3.center)--(l1.center);
\end{tikzpicture}
$$
Clearly, the leftmost triangle of $T$ is either labeled by the first integer in ${\bf i}$ or labeled by the first integer in ${\bf j}$.

If $T$ is a triangulation of $S_{\bf j}^{\bf i}(A)$ such that the leftmost  triangle of $T$ is labeled by the first integer in ${\bf j}$, then we define $$L_d(S_{\bf j}^{\bf i}(A), T):=(L_d(S_{\bf j}^{\bf i}(A)), T^\prime),$$ where $T^\prime$ is obtained from $T$ by the following replacement for the leftmost triangle of $T$.
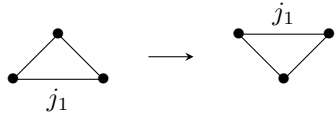
$$
\begin{tikzpicture}[scale=0.60]
\node (v1) at (0,0) {$\bullet$};
\node (v2) at (-1,-1) {$\bullet$};
\node (v3) at (1,-1) {$\bullet$};
\draw (v1.center)--(v2.center)-- node[below]{$j_1$}(v3.center)--(v1.center);
\draw[-stealth](2,-0.5)--(3,-0.5);
\node (l1) at (5,-1) {$\bullet$};
\node (l2) at (4,0) {$\bullet$};
\node (l3) at (6,0) {$\bullet$};
\draw (l1.center)--(l2.center)-- node[above]{$j_1$}(l3.center)--(l1.center);
\end{tikzpicture}$$

If $T$ is a triangulation of $S_{\bf j}^{\bf i}(A)$ such that the leftmost  triangle of $T$ is labeled by the first integer in ${\bf i}$, then we define  $$L_u(S_{\bf j}^{\bf i}(A), T):=(L_u(S_{\bf j}^{\bf i}(A)), T^\prime),$$ where $T^\prime$ is obtained from $T$ by the following replacement for the leftmost triangle of $T$.
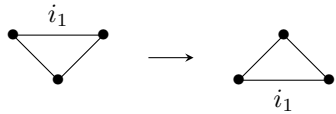
$$   \begin{tikzpicture}[scale=0.60]
\draw[-stealth](2,-0.5)--(3,-0.5);
\node (l1) at (0,-1) {$\bullet$};
\node (l2) at (-1,0) {$\bullet$};
\node (l3) at (1,0) {$\bullet$};
\draw (l1.center)--(l2.center)-- node[above]{$i_1$}(l3.center)--(l1.center);
\node (v1) at (5,0) {$\bullet$};
\node (v2) at (4,-1) {$\bullet$};
\node (v3) at (6,-1) {$\bullet$};
\draw (v1.center)--(v2.center)-- node[below]{$i_1$}(v3.center)--(v1.center);
\end{tikzpicture}$$

One can similarly define $R_d(S_{\bf j}^{\bf i}(A), T)$ for $T$ whose rightmost triangle is labeled by the last integer in ${\bf j}$ and  $R_u(S_{\bf j}^{\bf i}(A), T)$ for $T$ whose rightmost triangle is labeled by the last integer in ${\bf i}$.

We remark that the definitions of $L_d, L_u, R_d, R_u$ follow the ideas of the definitions of reflections in \cite{Shen-Weng-2021}*{Section 2.3}.

\begin{example}
Take $(S_{\bf j}^{\bf i}(A), T)$ as the one in Example \ref{ex:tri}, then
$L_d(S_{\bf j}^{\bf i}(A), T)$ and its string diagram are given in Figure \ref{fig:2}. If we further apply $L_u$ to $L_d(S_{\bf j}^{\bf i}(A), T)$, then we go back to $(S_{\bf j}^{\bf i}(A),T)$.

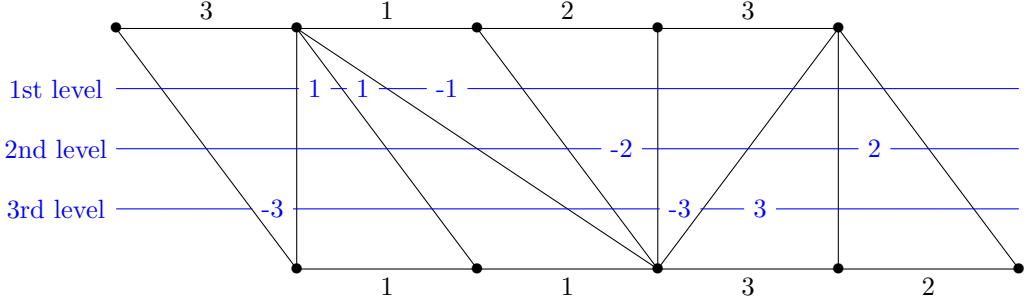
\begin{figure}
  \centering
\begin{tikzpicture}[scale=0.80]
  \node (A1) at (3,0) {$\bullet$};
    \node (A2) at (6,0) {$\bullet$};
       \node (A3) at (9,0) {$\bullet$};
          \node (A4) at (12,0) {$\bullet$};
          \node (B1) at (0,0) {$\bullet$};

\draw (A1.center) --node [above] {$1$} (A2.center);
\draw (A2.center) --node [above] {$2$} (A3.center);
\draw(A3.center) --node [above] {$3$} (A4.center);
  \draw (B1.center) --node [above] {$3$} (A1.center);

    \node (B2) at (3,-4) {$\bullet$};
       \node (B3) at (6,-4) {$\bullet$};
          \node (B4) at (9,-4) {$\bullet$};
              \node (B5) at (12,-4) {$\bullet$};
                  \node (B6) at (15,-4) {$\bullet$};

                  \draw (B2.center) --node [below] {$1$} (B3.center);
                  \draw (B3.center) --node [below] {$1$} (B4.center);
                  \draw (B4.center) --node [below] {$3$} (B5.center);
                  \draw (B5.center) --node [below] {$2$} (B6.center);
                  
                  \draw (B1.center) -- (B2.center);
                  \draw (A1.center) -- (B2.center);
                  \draw (A1.center) -- (B3.center);
                  \draw (A1.center) -- (B4.center);
                  \draw (A2.center) -- (B4.center);
                  \draw (A3.center) -- (B4.center);
                  \draw (A4.center) -- (B4.center);
                  \draw (A4.center) -- (B5.center);
                  \draw (A4.center) -- (B6.center);
                  
              \node (l1) at (3.3,-1) {\textcolor{blue}{1}};
                   \node (l2) at (4.1,-1) {\textcolor{blue}{1}};
                   \node (l3) at (5.5,-1) {\textcolor{blue}{-1}};
                   \draw[blue] (0,-1)--(l1)--(l2)--(l3)--(15,-1);
                   
                   \node (m1) at (8.4,-2) {\textcolor{blue}{-2}};
                   \node (m2) at (12.6,-2) {\textcolor{blue}{2}};
                   \draw[blue] (0,-2)--(m1)--(m2)--(15,-2);
                   \node (n1) at (2.6,-3) {\textcolor{blue}{-3}};
                   \node (n2) at (9.37,-3) {\textcolor{blue}{-3}};
                   \node (n3) at (10.7,-3) {\textcolor{blue}{3}};
                   \draw[blue] (0,-3)--(n1)--(n2)--(n3)--(15,-3);
                   \node at (-1,-1) {\textcolor{blue}{1st level}};
                    \node at (-1,-2) {\textcolor{blue}{2nd level}};
                     \node at (-1,-3) {\textcolor{blue}{3rd level}};
\end{tikzpicture}
 \caption{After applying $L_d$ to $(S_{\bf j}^{\bf i}(A), T)$.}\label{fig:2}
 \end{figure}

\end{example}

\begin{lemma}\label{pro:move} Given an $N\in\{L_u,L_d,R_u,R_d\}$. If $N(S_{\bf j}^{\bf i}(A), T)$ is defined for $T$,  then the exchange matrices of $(S_{\bf j}^{\bf i}(A), T)$ and  $N(S_{\bf j}^{\bf i}(A), T)$ are the same.
\end{lemma}
\begin{proof}
The proof here is essentially contained in the proof of \cite{Shen-Weng-2021}*{Proposition 4.2}. We prove the case $N=L_d$. By the similar arguments, one can prove the other cases.

Let $B$ be the exchange matrix of $(S_{\bf j}^{\bf i}(A), T)$. Recall that $B$ is defined by 
 $$B=\sum_{\text{nodes}\; m\;\text{in}\; T}B^{(m)}.$$

 A basic observation is that there is no closed strings in the string diagram of $T$ intersect the leftmost triangle of $T$ transversely. Thanks to this observation and by the definition of $B^{(m)}$, we know that the following replacement 
 $$
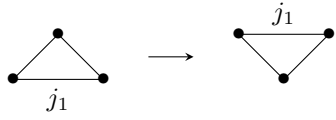
\begin{tikzpicture}[scale=0.60]
\node (v1) at (0,0) {$\bullet$};
\node (v2) at (-1,-1) {$\bullet$};
\node (v3) at (1,-1) {$\bullet$};
\draw (v1.center)--(v2.center)-- node[below]{$j_1$}(v3.center)--(v1.center);
\draw[-stealth](2,-0.5)--(3,-0.5);
\node (l1) at (5,-1) {$\bullet$};
\node (l2) at (4,0) {$\bullet$};
\node (l3) at (6,0) {$\bullet$};
\draw (l1.center)--(l2.center)-- node[above]{$j_1$}(l3.center)--(l1.center);
\end{tikzpicture}$$
 for the leftmost triangle of $T$ preserves each $B^{(m)}$. Hence, the exchange matrices of $(S_{\bf j}^{\bf i}(A), T)$ and  $N(S_{\bf j}^{\bf i}(A), T)=L_d(S_{\bf j}^{\bf i}(A), T)$ are the same.
\end{proof}

Let $({\bf i},{\bf j})$ be a pair of $[1,n]$-sequences, 
say 
$${\bf i}=(i_1,\dots,i_{\ell(\bf i)})\;\;\;\text{and}\;\;\;{\bf j}=(j_1,\dots,j_{\ell(\bf j)}).$$
We denote by ${\bf i}^{-1}:=(i_{\ell({\bf i})},\cdots,i_1)$ and ${\bf i}^{-1}\circ {\bf j}:=(i_{\ell({\bf i})},\cdots,i_1,j_1,\cdots,j_{\ell({\bf j})})$.

\begin{proposition}
\label{cor:move}
Let $B$ be the exchange matrix of a triangulation $T$ of $S_{\bf j}^{\bf i}(A)$ and $B_{{\bf i}^{-1}\circ {\bf j}}$ the exchange matrix of the (unique) triangulation $T_{{\bf i}^{-1}\circ {\bf j}}$ of $S_{{\bf i}^{-1}\circ {\bf j}}^{\emptyset}(A)$. Then $B$ and $B_{{\bf i}^{-1}\circ {\bf j}}$ can be obtained from each other by a sequence of mutations.
\end{proposition}

\begin{proof}
If ${\bf i}=\emptyset$, then $S_{\bf j}^{\bf i}(A)=S_{{\bf i}^{-1}\circ {\bf j}}^{\emptyset}(A)=S_{\bf j}^{\emptyset}(A)$. Then by the fact that $S_{\bf j}^{\emptyset}(A)$ has a unique triangulation, we get $T=T_{{\bf i}^{-1}\circ {\bf j}}$ and $B=B_{{\bf i}^{-1}\circ {\bf j}}$.

If ${\bf i}\neq \emptyset$, then there exists a triangulation $M$ of $S_{\bf j}^{\bf i}(A)$ such that the leftmost triangle of $M$ is labeled by the first integer in ${\bf i}$. Denote by $B_M$ the exchange matrix of $M$. By Proposition \ref{pro:sw2}, we know that $B$ and $B_M$ can be obtained from each other by a sequence of mutations.

Notice that $R_u(S_{\bf j}^{\bf i}(A),M)$ is defined for $M$ by the choice of $M$. Then
by Lemma \ref{pro:move}, we know that $B_M$ is also the exchange matrix of $R_u(S_{\bf j}^{\bf i}(A),M)=(S_{{\bf j}^\prime}^{{\bf i}^\prime}(A),M^\prime)$, where ${\bf i}^\prime=(i_2,\cdots,i_{\ell({\bf i})})\;\text{and}\;{\bf j}^\prime=(i_1,j_1,\dots,j_{\ell({\bf j})})$. So we reduce $({\bf i},{\bf j})$ to $({\bf i}^\prime,{\bf j}^\prime)$.

By repeating the above process, we get that $B$ and $B_{{\bf i}^{-1}\circ {\bf j}}$ can be obtained from each other by a sequence of mutations.
\end{proof}

We remark that Lemma \ref{pro:move} is very useful in practice, because some interesting mutation sequences can be visualized thanks to Lemma \ref{pro:move} and Proposition \ref{pro:sw1}.  We give an example to explain the ideas.
\begin{example}
Keep the notations in Example \ref{ex:gls}. The vertices of $Q_{\bf j}$ is $$I=\{{1\choose 1}\prec {1\choose 2}\prec{1\choose 3}\prec{2\choose 1}\prec{2\choose 2}\prec{2\choose 3}\prec{3\choose 1}\}$$ 
We consider the mutation sequence  $$\overset{\leftarrow}{\mu}=\mu_{{1\choose 3}}\circ \mu_{{1\choose 2}}\circ \mu_{{1\choose 1}}.$$
We want to realize each mutation in the above mutation sequence as a flip in a suitable triangulation of some trapezoid. 

The triangulation $T$ in Example \ref{ex:gls} gives the quiver $Q_{\bf j}$. However, the triangulation in Example \ref{ex:gls} clearly can not be flipped. Thanks to Lemma \ref{pro:move}, we know that $Q_{\bf j}$ can be also viewed as the quiver of $L_d(S_{\bf j}^\emptyset(A), T)=(L_d(S_{\bf j}^\emptyset(A)), T^\prime)$, which is as follows:

$$
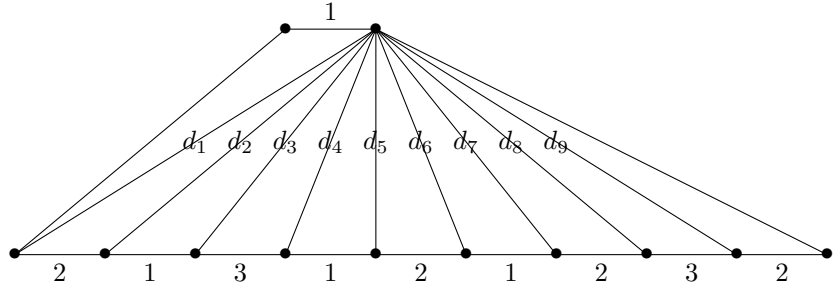
\begin{tikzpicture}[scale=0.60]
  
       \node (a) at (0,0) {$\bullet$};
       \node (b1) at (-2,0) {$\bullet$};
         \node (b2) at (-8,-5) {$\bullet$};
           \node (b3) at (-6,-5) {$\bullet$};
             \node (b4) at (-4,-5) {$\bullet$};
               \node (b5) at (-2,-5) {$\bullet$};
                 \node (b6) at (0,-5) {$\bullet$};
                  \node (b7) at (2,-5) {$\bullet$};
                  \node (b8) at (4,-5) {$\bullet$};
           \node (b9) at (6,-5) {$\bullet$};
             \node (b10) at (8,-5) {$\bullet$};
               \node (b11) at (10,-5) {$\bullet$};
               
\draw (b1.center) -- (b2.center);    
\draw (b2.center) --node [below] {$2$} (b3.center);  
\draw (b3.center) --node [below] {$1$} (b4.center);    
\draw (b4.center) --node [below] {$3$} (b5.center); 
\draw (b5.center) --node [below] {$1$} (b6.center);
\draw (b6.center) --node [below] {$2$} (b7.center);
\draw (b7.center) --node [below] {$1$} (b8.center);
\draw (b8.center) --node [below] {$2$} (b9.center);
\draw (b9.center) --node [below] {$3$} (b10.center);
\draw (b10.center) --node [below] {$2$} (b11.center);
               \draw (a.center) --node [above] {$1$} (b1.center);
                 \draw (a.center) --node {$d_1$} (b2.center);
                   \draw (a.center) --node {$d_2$} (b3.center);
                     \draw (a.center) --node {$d_3$} (b4.center);
                       \draw (a.center) -- node {$d_4$}(b5.center);
                         \draw (a.center) -- node {$d_5$}(b6.center);
                           \draw (a.center) --node {$d_6$} (b7.center);
                             \draw (a.center) --node {$d_7$} (b8.center);
                               \draw (a.center) --node {$d_8$} (b9.center);
                                 \draw (a.center) --node {$d_9$} (b10.center);
                                  \draw (a.center) -- (b11.center);
            
\end{tikzpicture}
$$

In the above triangulation of $L_d(S_{\bf j}^\emptyset(A))$, we flip the diagonals  $$(d_1,d_2,d_3,d_4,d_5,d_6,d_7,d_8,d_9)$$ in the given order and this corresponds to the sequence  $$({\rm id},\mu_{{1\choose 1}},{\rm id},\mu_{{1\choose 2}},{\rm id},\mu_{{1\choose 3}},{\rm id},{\rm id},{\rm id})$$ 
acting on $Q_{\bf j}$ by Proposition \ref{pro:sw1}. Thus the mutation sequence $\overset{\leftarrow}{\mu}=\mu_{{1\choose 3}}\circ \mu_{{1\choose 2}}\circ \mu_{{1\choose 1}}$ can be visualized. After the sequence of flips, the new triangulation $T^{\prime\prime}$ of $L_d(S_{\bf j}^\emptyset(A))$ is as follows:

$$
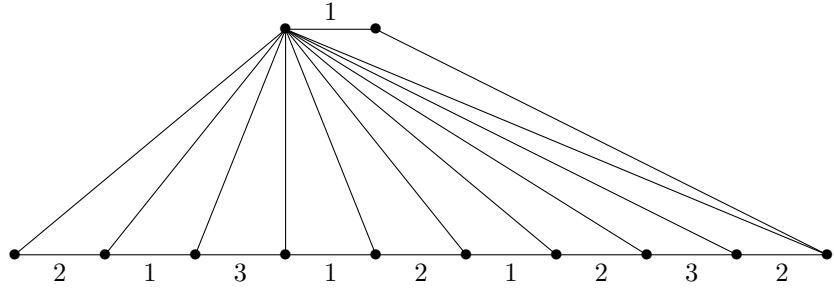
\begin{tikzpicture}[scale=0.60]
  
       \node (a) at (0,0) {$\bullet$};
       \node (b1) at (-2,0) {$\bullet$};
         \node (b2) at (-8,-5) {$\bullet$};
           \node (b3) at (-6,-5) {$\bullet$};
             \node (b4) at (-4,-5) {$\bullet$};
               \node (b5) at (-2,-5) {$\bullet$};
                 \node (b6) at (0,-5) {$\bullet$};
                  \node (b7) at (2,-5) {$\bullet$};
                  \node (b8) at (4,-5) {$\bullet$};
           \node (b9) at (6,-5) {$\bullet$};
             \node (b10) at (8,-5) {$\bullet$};
               \node (b11) at (10,-5) {$\bullet$};
               
\draw (b1.center) -- (b2.center);    
\draw (b2.center) --node [below] {$2$} (b3.center);  
\draw (b3.center) --node [below] {$1$} (b4.center);    
\draw (b4.center) --node [below] {$3$} (b5.center); 
\draw (b5.center) --node [below] {$1$} (b6.center);
\draw (b6.center) --node [below] {$2$} (b7.center);
\draw (b7.center) --node [below] {$1$} (b8.center);
\draw (b8.center) --node [below] {$2$} (b9.center);
\draw (b9.center) --node [below] {$3$} (b10.center);
\draw (b10.center) --node [below] {$2$} (b11.center);
               \draw (a.center) --node [above] {$1$} (b1.center);
               
                   \draw (b1.center) -- (b3.center);
                     \draw (b1.center) -- (b4.center);
                       \draw (b1.center) -- (b5.center);
                         \draw (b1.center) -- (b6.center);
                           \draw (b1.center) -- (b7.center);
                             \draw (b1.center) -- (b8.center);
                               \draw (b1.center) -- (b9.center);
                                 \draw (b1.center) -- (b10.center);
                                  \draw (b1.center) -- (b11.center);
                                  \draw (a.center) -- (b11.center);
\end{tikzpicture}
$$

If we further apply $R_u$ to $(L_d(S_{\bf j}^\emptyset(A)), T^{\prime\prime})$, we get the following triangulation of $R_uL_d(S_{\bf j}^\emptyset(A))$.

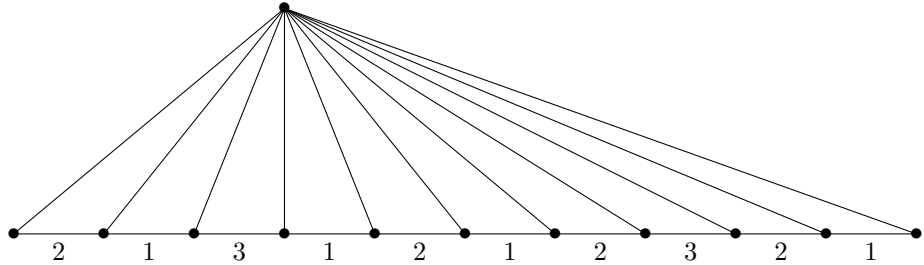
$$
\begin{tikzpicture}[scale=0.60]
  
       \node (a) at (12,-5) {$\bullet$};
       \node (b1) at (-2,0) {$\bullet$};
         \node (b2) at (-8,-5) {$\bullet$};
           \node (b3) at (-6,-5) {$\bullet$};
             \node (b4) at (-4,-5) {$\bullet$};
               \node (b5) at (-2,-5) {$\bullet$};
                 \node (b6) at (0,-5) {$\bullet$};
                  \node (b7) at (2,-5) {$\bullet$};
                  \node (b8) at (4,-5) {$\bullet$};
           \node (b9) at (6,-5) {$\bullet$};
             \node (b10) at (8,-5) {$\bullet$};
               \node (b11) at (10,-5) {$\bullet$};
               
\draw (b1.center) -- (b2.center);    
\draw (b2.center) --node [below] {$2$} (b3.center);  
\draw (b3.center) --node [below] {$1$} (b4.center);    
\draw (b4.center) --node [below] {$3$} (b5.center); 
\draw (b5.center) --node [below] {$1$} (b6.center);
\draw (b6.center) --node [below] {$2$} (b7.center);
\draw (b7.center) --node [below] {$1$} (b8.center);
\draw (b8.center) --node [below] {$2$} (b9.center);
\draw (b9.center) --node [below] {$3$} (b10.center);
\draw (b10.center) --node [below] {$2$} (b11.center);
               \draw (a.center) -- (b1.center);
               
                   \draw (b1.center) -- (b3.center);
                     \draw (b1.center) -- (b4.center);
                       \draw (b1.center) -- (b5.center);
                         \draw (b1.center) -- (b6.center);
                           \draw (b1.center) -- (b7.center);
                             \draw (b1.center) -- (b8.center);
                               \draw (b1.center) -- (b9.center);
                                 \draw (b1.center) -- (b10.center);
                                  \draw (b1.center) -- (b11.center);
                                  \draw (a.center) --node [below] {$1$} (b11.center);
\end{tikzpicture}
$$
\end{example}

\subsection{Exchange matrices from string diagrams are in the class in \texorpdfstring{$\mathcal P'$}{Lg}} In this section, we give the proof of our main theorem.

\begin{lemma}\label{lem:empty}
Let $T_{\bf j}$ be the (unique) triangulation of $S_{\bf j}^{\emptyset}(A)$ and $D(T_{\bf j})$ the string diagram of $T_{\bf j}$. Let $j_1$ be the first integer in ${\bf j}$ and $m\geq 0$ the number of closed strings on the $j_1$-th level of $D(T_{\bf j})$. Suppose that $m\geq 1$ and put  $$\overset{\leftarrow}{\mu}:=\mu_{j_1 \choose 2}\circ\cdots \circ \mu_{j_1 \choose m-1}\circ \mu_{j_1 \choose m}.$$
Let $B=(b_{ij})_{i,j\in I}$ be the exchange matrix of $T_{\bf j}$, then the column of $B^\prime:=\overset{\leftarrow}{\mu}(B)=(b_{ij}^\prime)_{i,j\in I}$ indexed by the closed string ${j_1 \choose 1}$ is a non-negative vector.
\end{lemma}
\begin{proof}

If $m=1$, then $\overset{\leftarrow}{\mu}={\rm id}$ and $B^\prime=B$. In this case,  the result follows from the definition of $B$. So we can assume that $m\geq 2$.

Denote by $I_{j_1}$ the set of closed strings on the $j_1$-th level and by $z_i:={j_1\choose i}$ the $i$-th closed string in $I_{j_1}$. In the mutation sequence  $$\overset{\leftarrow}{\mu}=\mu_{j_1 \choose 2}\circ\cdots \circ \mu_{j_1 \choose m-1}\circ \mu_{j_1 \choose m}=\mu_{z_2}\cdots \mu_{z_{m-1}}\mu_{z_m},$$ the vertices outside $I_{j_1}$ are not mutated. So we can just freeze the vertices outside $I_{j_1}$. This means that we can just set $b_{xz}=0$ in $B$ for any $z\notin I_{j_1}$. Now $B$ takes the following form:
 
 $$B=\begin{bmatrix}B_{j_1}^\circ&0\\C&0\end{bmatrix},$$
 where $B_{j_1}:=\begin{bmatrix}B_{j_1}^\circ\\ C\end{bmatrix}$  and $B_{j_1}^\circ$ are the $I\times I_{j_1}$ and $I_{j_1}\times I_{j_1}$ submatrices of $B$ respectively. Notice that the usual quiver of $B_{j_1}^\circ$ is as follows:
 $$\xymatrix{z_1& z_2\ar[l]&z_3\ar[l]& \cdots\ar[l]&z_{m-1}\ar[l]&z_m\ar[l]}$$

 Recall that we want to prove that the column of $B^\prime=\overset{\leftarrow}{\mu}(B)=(b_{ij}^\prime)_{i,j\in I}$ indexed by the closed string $z_1={j_1 \choose 1}$ is a non-negative vector. It suffices to prove $b_{xz_1}^\prime\geq 0$ for each $x\in I$. We distinguish it into two cases: Case (i): $x\in I_{j_1}$ and Case (ii): $x\notin I_{j_1}$.
 
 Case (i): Assume $x\in I_{j_1}$. We want to prove $b_{xz_1}^\prime\geq 0$. By $x\in I_{j_1}$, we can just set $C=0$ and then $B$ takes the form $$B=\begin{bmatrix}B_{j_1}^\circ&0\\0&0\end{bmatrix}.$$
 After the mutation sequence $\overset{\leftarrow}{\mu}=\mu_{z_2}\cdots\mu_{z_{m-1}} \mu_{z_m}$, the usual quiver of $B_{j_1}^\circ$ is mutated to the following quiver.
 
 $$\xymatrix{z_1\ar[r]& z_2&z_3\ar[l]& \cdots\ar[l]&z_{m-1}\ar[l]&z_m\ar[l]}$$
 So we have $b_{xz_1}^\prime\in\{0,1\}$ and thus $b_{xz_1}^\prime\geq 0$ holds in this case.

 Case (ii): Assume $x\notin I_{j_1}$, say, $x$ is on the $k$-th level ($k\neq j_1$). We want to prove $b_{xz_1}^\prime\geq 0$.
  In this case, we can set the rows of $C$ not indexed by $x$ to the zero vector. Then  $B$ takes the form $$B=\begin{bmatrix}B_{j_1}^\circ&0\\ \alpha&0\\0&0\end{bmatrix},$$
 where $\alpha$ is the $x$-th row of $C$. 
 
Notice that in the coloured quiver of the exchange matrix of $(S_{\bf j}^\emptyset(A), T_{\bf j})$, if the inclined arrows between the levels $j_1$ and  $k$ exist, then they look like as follows.
 \[\xymatrixrowsep{5mm}
\xymatrixcolsep{5mm}
 
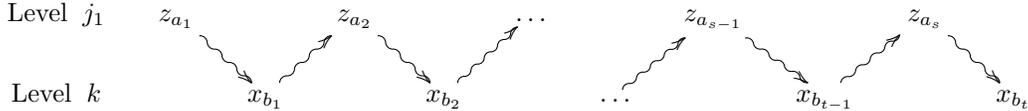
\xymatrix{{\rm Level}\;\; j_1& z_{a_1}\ar@{~>}[rd]&&z_{a_2}\ar@{~>}[rd]&&\dots&&z_{a_{s-1}}\ar@{~>}[rd]&&z_{a_s}\ar@{~>}[rd]
 \\ 
 {\rm Level}\;\; k&&x_{b_1}\ar@{~>}[ru]&&x_{b_2}\ar@{~>}[ru]&&\dots\ar@{~>}[ru]&&x_{b_{t-1}}\ar@{~>}[ru]&&x_{b_t}
 }
 \]
 (The inclined arrow $\xymatrix{z_{a_s}\ar@{~>}[r]&x_{b_t}}$ does not exist necessarily. But the leftmost inclined arrow between level $j_1$ and level $k$ always start at the level $j_1$.) Since $x$ is a closed string on the level $k$ and $\alpha$ is the $x$-th row of $C$, we get that $\alpha$ has at most two non-zero entries.

 If $\alpha=0$, then after the mutation sequence $\overset{\leftarrow}{\mu}=\mu_{z_2}\cdots\mu_{z_{m-1}} \mu_{z_m}$, we have   $$\overset{\leftarrow}{\mu}\begin{bmatrix}B_{j_1}^\circ\\ \alpha \end{bmatrix}=\overset{\leftarrow}{\mu}\begin{bmatrix}B_{j_1}^\circ\\ 0 \end{bmatrix}=\begin{bmatrix}\ast \\ 0 \end{bmatrix}.$$
 Thus $b_{xz_1}^\prime =0$ and $b_{xz_1}^\prime \geq 0$ holds in this case.
 
It remains to deal with the case $\alpha\neq 0$.  Because the matrix $\begin{bmatrix}B_{j_1}^\circ&-\alpha^{\rm T}\\ \alpha&0 \end{bmatrix}$ is actually a skew-symmetric matrix,  we can represent it using a usual quiver. 

If $\alpha$ has exactly one non-zero entry, then the usual quiver of $\begin{bmatrix}B_{j_1}^\circ&-\alpha^{\rm T}\\ \alpha&0 \end{bmatrix}$ takes the form

 $$\xymatrix{z_1&z_2\ar[l]&\ar[l]\cdots \ar[l]&z_p\ar[l]\ar[d]|{-a_{kj_1}}& z_{p+1}\ar[l]&\ar[l]\cdots\ar[l]&z_q\ar[l]&z_{q+1}\ar[l]&\ar[l]\dots\ar[l]&z_m\ar[l]\\ &&&x&&&&&&}
  $$
where $a_{kj_1}<0$ is the $(k,j_1)$-entry of the Cartan matrix $A$ and the $-a_{kj_1}$ arrows from $z_p$ to $x$ correspond to the non-zero entry of $\alpha$.

If $p=1$, then after the mutation sequence $\overset{\leftarrow}{\mu}=\mu_{z_2}\cdots\mu_{z_{m-1}} \mu_{z_m}$, the usual quiver of $\begin{bmatrix}B_{j_1}^\circ&-\alpha^{\rm T}\\ \alpha&0 \end{bmatrix}$ is mutated to the following quiver.

 $$\xymatrix{z_1\ar[r]\ar[d]|{-a_{kj_1}}\ar[r]&z_2&\cdots \ar[l]&z_p\ar[l]& z_{p+1}\ar[l]&\ar[l]\cdots\ar[l]&z_q\ar[l]&z_{q+1}\ar[l]&\ar[l]\dots\ar[l]&z_m\ar[l]\\ x&&&&&&&&&}
  $$
 Thus $b_{xz_1}^\prime=-a_{kj_1}>0$ and $b_{xz_1}^\prime\geq 0$ holds in this case.
 
 If $p\neq 1$, then 
 after the mutation sequence $\overset{\leftarrow}{\mu}=\mu_{z_2}\cdots\mu_{z_{m-1}} \mu_{z_m}$, the usual quiver of $\begin{bmatrix}B_{j_1}^\circ&-\alpha^{\rm T}\\ \alpha&0 \end{bmatrix}$ is mutated to the following quiver.
 
  $$\xymatrix{z_1\ar[r]&z_2&\ar[l]\cdots \ar[l]&z_p\ar[l]& z_{p+1}\ar[l]&\ar[l]\cdots\ar[l]&z_q\ar[l]&z_{q+1}\ar[l]&\ar[l]\dots\ar[l]&z_m\ar[l]\\ &&&x\ar[u]|{-a_{kj_1}}&&&&&&}
  $$
 Thus $b_{xz_1}^\prime =0$ and $b_{xz_1}^\prime \geq 0$ holds in this case.
 
 If $\alpha$ has exactly two non-zero entries, then the usual quiver of $\begin{bmatrix}B_{j_1}^\circ&-\alpha^{\rm T}\\ \alpha&0 \end{bmatrix}$ takes the form
 
$$ \xymatrix{z_1&z_2\ar[l]&\ar[l]\cdots \ar[l]&z_p\ar[l]\ar[rrd]|{-a_{kj_1}}& z_{p+1}\ar[l]&\ar[l]\cdots\ar[l]&z_q\ar[l]&z_{q+1}\ar[l]&\ar[l]\dots\ar[l]&z_m\ar[l]\\ &&&&&x\ar[ru]|{-a_{kj_1}}&&&&}
$$
 After the mutation sequence $\overset{\leftarrow}{\mu}=\mu_{z_2}\cdots\mu_{z_{m-1}} \mu_{z_m}$,
 the usual quiver of $\begin{bmatrix}B_{j_1}^\circ&-\alpha^{\rm T}\\ \alpha&0 \end{bmatrix}$ is mutated to the following quiver.
 
 $$ \xymatrix{z_1\ar[r]&z_2&\ar[l]\cdots \ar[l]&z_p\ar[l]& z_{p+1}\ar[l]\ar[rd]|{-a_{kj_1}}&\ar[l]\cdots\ar[l]&z_q\ar[l]&z_{q+1}\ar[l]&\ar[l]\dots\ar[l]&z_m\ar[l]\\ &&&&&x\ar[rru]|{-a_{kj_1}}&&&&}
$$
 (In case $q=m$, the $-a_{kj_1}$ arrows from $x$ to $z_{q+1}=z_{m+1}$ do not exist.) Thus $b_{xz_1}^\prime =0$ and $b_{xz_1}^\prime \geq 0$ holds in this case.

 By the results in Case (i) and Case (ii), we get $b_{xz_1}^\prime\geq 0$ for any $x\in I$. So the column of $B^\prime:=\overset{\leftarrow}{\mu}(B)=(b_{ij}^\prime)_{i,j\in I}$ indexed by the closed string $z_1={j_1 \choose 1}$ is a non-negative vector.
\end{proof}

By the similar arguments, we can prove the following dual result.

\begin{lemma}
Let $T_{\bf j}$ be the (unique) triangulation of $S_{\bf j}^{\emptyset}(A)$ and $D(T_{\bf j})$ the string diagram of $T_{\bf j}$. Let $j_\ell$ be the last integer in ${\bf j}$ and $m^\prime\geq 0$ the number of closed strings on the $j_\ell$-th level of $D(T_{\bf j})$. Suppose that $m^\prime\geq 1$ and put  $$(\overset{\leftarrow}{\mu})^\prime:=\mu_{j_\ell \choose m^\prime-1}\circ\cdots \circ \mu_{j_\ell \choose 2}\circ \mu_{j_\ell \choose 1}.$$
Let $B=(b_{ij})_{i,j\in I}$ be the exchange matrix of $T_{\bf j}$, then the column of $B^\prime:=(\overset{\leftarrow}{\mu})^\prime(B)=(b_{ij}^\prime)_{i,j\in I}$ indexed by the closed string ${j_\ell \choose m^\prime}$ is a non-positive vector.
\end{lemma}

\begin{example}
Keep the notations in Example \ref{ex:gls}. After the mutation sequence $$\overset{\leftarrow}{\mu}:=\mu_{j_1 \choose 2}\circ\cdots \circ \mu_{j_1 \choose m-1}\circ \mu_{j_1 \choose m}=\mu_{1\choose 2}\circ\mu_{1\choose 3},$$
the quiver $Q_{\bf j}$ in Example \ref{ex:gls} is mutated to the following quiver.
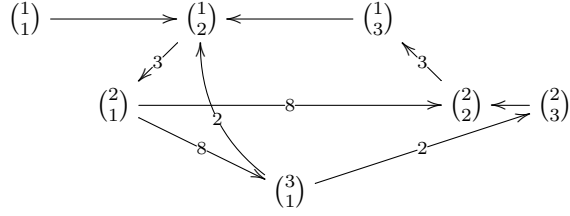
\[\xymatrixrowsep{5mm}
\xymatrixcolsep{5mm}
\xymatrix{{1 \choose 1}\ar[rr] &&{1 \choose 2}\ar[ld]|{3} & &{1 \choose 3}\ar[ll]&& \\
 &{2 \choose 1}\ar[rrrr]|{8}\ar[rrd]|8 && &&{2 \choose 2}\ar[lu]|3 &{2 \choose 3}\ar[l] \\
  &&&{3 \choose 1}\ar@/_{-1pc}/[luu]|2\ar[rrru]|2 &&&}
  \]
One can see that the vertex ${1\choose 1}$ becomes a source vertex.

After the mutation sequence $$(\overset{\leftarrow}{\mu})^\prime:=\mu_{j_\ell \choose m^\prime-1}\circ\cdots \circ \mu_{j_\ell \choose 2}\circ \mu_{j_\ell \choose 1}=\mu_{2\choose 2}\circ\mu_{2\choose 1},$$
the quiver $Q_{\bf j}$ in Example \ref{ex:gls} is mutated to the following quiver.
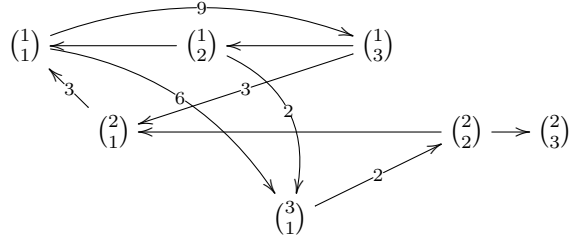
\[\xymatrixrowsep{5mm}
\xymatrixcolsep{5mm}
\xymatrix{{1 \choose 1}\ar@/_{-1.2pc}/[rrrr]|9\ar@/_{-1.3pc}/[rrrdd]|6 &&{1 \choose 2}\ar[ll]\ar@/_{-1.5pc}/[rdd]|2 & &{1 \choose 3}\ar[ll]\ar[llld]|3&& \\
 &{2 \choose 1}\ar[lu]|3 && &&{2 \choose 2} \ar[llll]\ar[r]&{2 \choose 3} \\
  &&&{3 \choose 1}\ar[rru]|2 &&&}
\]
One can see that the vertex ${2\choose 3}$ becomes a sink vertex.
\end{example}

\begin{theorem}\label{thm:string}
Let $B$ be the exchange matrix of a triangulation $T$ of $S_{\bf j}^{\bf i}(A)$. Then $B$ is in the class $\mathcal P'$.
\end{theorem}
\begin{proof}
We first deal with the case ${\bf i}=\emptyset$. Then we use the result for $S_{\bf j}^{\emptyset}(A)$ to deduce the result for the general case.

Now we assume ${\bf i}=\emptyset$.  We prove the result for $S_{\bf j}^{\emptyset}(A)$ by induction on the length $\ell({\bf j})$ of  $${\bf j}=(j_1,j_2,\cdots,j_{\ell({\bf j})}).$$

If $\ell({\bf j})=1$, then the string diagram of $(S_{\bf j}^{\emptyset}(A), T)$ has no closed strings. So the vertex set of $B$ is empty. In this case, $B$ is in the class $\mathcal P^\prime$ following the usual convention.

Assume by induction that the result holds for  $(S_{{\bf j}^\prime}^\emptyset(A), T^\prime)$ with $\ell({\bf j}^\prime)<\ell({\bf j})$.

For ${\bf j}=(j_1,j_2,\cdots,j_{\ell({\bf j})})$, we denote by ${\bf j}_{\geq 2}:=(j_2,\dots,j_{\ell({\bf j})})$. Let $T_{\geq 2}$ be the (unique) triangulation of $S_{{\bf j}_{\geq 2}}^\emptyset(A)$ and $B_{\geq 2}$ the exchange matrix of $T_{\geq 2}$. Since $\ell({\bf j}_{\geq 2})<\ell({\bf j})$, we can apply the inductive hypothesis and get that $B_{\geq 2}$ is in the class $\mathcal P^\prime$.

Let $D(T)$ be the string diagram of $(S_{\bf j}^{\emptyset}(A), T)$
and denote by $m\geq 0$ the number of closed string on the $j_1$-th level of $D(T)$. If $m=0$, then $B=B_{\geq 2}$. Since $B_{\geq 2}$ is in the class $\mathcal P^\prime$, so is $B$. 

So we can assume $m\geq 1$. In this case, ${j_1\choose 1}$ is a non-empty closed string and thus a vertex of $B$.
By the definition of $B$ and $B_{\geq 2}$, we know that $B_{\geq 2}$ is obtained from $B$ by deleting its row and column indexed by the closed string ${j_1\choose 1}$.

Let $\overset{\leftarrow}{\mu}=\mu_{j_1 \choose 2}\circ\cdots \circ \mu_{j_1 \choose m-1}\circ \mu_{j_1 \choose m}$ be the mutation sequence in Lemma \ref{lem:empty}. By Lemma \ref{lem:empty}, we know that $\overset{\leftarrow}{\mu}(B)$ is a source-sink extension of $\overset{\leftarrow}{\mu}(B_{\geq 2})$. Since $B_{\geq 2}$ is in the class $\mathcal P^\prime$ and the class $\mathcal P^\prime$ is closed under mutations and source-sink extension, we get $\overset{\leftarrow}{\mu}(B)$ is in the class $\mathcal P^\prime$ and thus $B$ is in the class $\mathcal P^\prime$. This completes the induction.

Hence, the exchange matrix $B$ of $(S_{\bf j}^{\emptyset}(A), T)$ is in the class $\mathcal P^\prime$ for any ${\bf j}$. Now we use this result to deduce the result for the case ${\bf i}\neq \emptyset$.

 By Proposition \ref{cor:move}, we know that the exchange matrix $B$ of $(S_{\bf j}^{\bf i}(A), T)$ can be obtained from the exchange matrix $B_{{\bf i}^{-1}\circ {\bf j}}$ of $(S_{{\bf i}^{-1}\circ {\bf j}}^\emptyset(A), T_{{\bf i}^{-1}\circ {\bf j}})$ by a sequence of mutations, where $T_{{\bf i}^{-1}\circ {\bf j}}$ is the unique triangulation of $S_{{\bf i}^{-1}\circ {\bf j}}^\emptyset(A)$. Notice that we have proved that $B_{{\bf i}^{-1}\circ {\bf j}}$ is in the class $\mathcal P^\prime$. Because the class $\mathcal P^\prime$ is closed under mutations, we get $B$ is in the class $\mathcal P^\prime$. This completes the proof.
\end{proof}

As applications of Theorem \ref{thm:string}, we have the following corollaries. We refer to \cites{keller_demonet_2020,bm_2020} for the notation of reddening sequences and  \cite{DWZ_2008} for the basic notions on quivers with potentials $(Q,W)$ and Jacobian algebras $J(Q,W)$.

\begin{corollary}[\cite{Shen-Weng-2021}*{Section 4}]
\label{app1}
Let $B$ be the exchange matrix of a triangulation $T$ of $S_{\bf j}^{\bf i}(A)$. Then $B$ has a reddening sequence.
\end{corollary}

\begin{proof}
This follows from Theorem \ref{thm:string} and \cite{bm_2020}*{Theorem 3.3}.
\end{proof}

\begin{corollary}
\label{app2}
Let $B$ be the exchange matrix of a triangulation $T$ of $S_{\bf j}^{\bf i}(A)$.
Suppose that $A$ is symmetric. In this case, $B$ is skew-symmetric and thus corresponds to a quiver $Q_B$. Then $Q_B$ has a unique non-degenerate potential $W_{B}$ (up to right equivalence) which is rigid and its Jacobian algebra $J(Q_{B}, W_{B})$ is finite-dimensional.
\end{corollary}
\begin{proof}
This follows from Theorem \ref{thm:string} and \cite{lad-2013}*{Theorem 4.6}.
\end{proof}
\begin{remark}
In the case of ${\bf i}=\emptyset$ and ${\bf j}$ corresponding to a reduced word of an element of the Weyl group of the Cartan matrix $A$, Buan-Iyama-Reiten-Smith  constructed a rigid potential $W$ on $Q_{B}$ in \cite{BIRS_2011}*{Section 6}. Then the above corollary tells us that $W$ and $W_{B}$ are the same up to right equivalence in this case.
\end{remark}

Let us end this paper with a very concrete example of Theorem \ref{thm:string}.
\begin{example}
Take the Cartan matrix $A=\begin{bmatrix} 2&-2&-3\\-2&2&-4\\-3&-4&2
\end{bmatrix}$ and  $${\bf j}=(j_1,\cdots,j_{10},j_{11})=(2,1,3,2,1,3,1,3,2,2,1).$$
The string diagram of $S_{\bf j}^{\emptyset}(A)$ can be found in \cite{Shen-Weng-2021}*{Page 51} and the usual quiver $Q_{\bf j}$ of the exchange matrix $B_{\bf j}$ of $S_{\bf j}^{\emptyset}(A)$ is the following quiver.

\[\xymatrixrowsep{5mm}
\xymatrixcolsep{5mm}
\xymatrix{ &{1\choose 1}\ar@/_{1pc}/[rdd]|3\ar[rrd]|2& &&{1\choose 2}\ar[lll]\ar@/_{1pc}/[rdd]|3 && {1\choose 3}\ar[ll]&
\\ 
{2\choose 1}\ar[ru]|2\ar[rrd]|4&& & {2\choose 2}\ar[lll]\ar[rrru]|2& && &{2\choose 3}\ar[llll]
\\
&&{3\choose 1}\ar[ru]|4 \ar@/_{-1.5pc}/[rruu]|3&& &{3\choose 2}\ar[lll]\ar@/_{-1pc}/[ruu]|3& &}
\]

By applying $\mu_{{2\choose 2}}\circ\mu_{{2\choose 3}}$ to the above quiver, we get the following quiver.

\[
\xymatrix{ &{1\choose 1}\ar@/_{1pc}/[rdd]|3\ar[rrrrrrd]|2\ar@/_{-2pc}/[rrrrr]|4& &&{1\choose 2}\ar[lll]\ar@/_{-1.5pc}/[rdd]|3 && {1\choose 3}\ar[ll]\ar@/_{-2pc}/[llld]|2&
\\ 
{2\choose 1}\ar[rrr]&& & {2\choose 2}\ar[llu]|2\ar[ld]|4& && &{2\choose 3}\ar@/_{1.4pc}/[llll]
\\
&&{3\choose 1} \ar@/_{-1.5pc}/[rruu]|3\ar[rrrruu]|8\ar[rrrrru]|4&& &{3\choose 2}\ar[lll]\ar[ruu]|3& &}
\]

By removing the source vertex ${2\choose 1}$ and applying $\mu_{{2\choose 3}}\circ\mu_{{2\choose 2}}$ to the resulting quiver, we get the quiver associated with $S_{{\bf j}_{\geq 2}}^\emptyset(A)$, where ${\bf j}_{\geq 2}=(j_2,\cdots,j_{10},j_{11})=(1,3,2,1,3,1,3,2,2,1).$ 
\[\xymatrixrowsep{5mm}
\xymatrixcolsep{5mm}
\xymatrix{ &{1\choose 1}\ar@/_{1pc}/[rdd]|3\ar[rrd]|2& &&{1\choose 2}\ar[lll]\ar@/_{1pc}/[rdd]|3 && {1\choose 3}\ar[ll]&
\\ 
&& & {2\choose 2}\ar[rrru]|2& && &{2\choose 3}\ar[llll]
\\
&&{3\choose 1}\ar[ru]|4 \ar@/_{-1.5pc}/[rruu]|3&& &{3\choose 2}\ar[lll]\ar@/_{-1pc}/[ruu]|3& &}
\]

By applying $\mu_{{1\choose 2}}\circ\mu_{{1\choose 3}}$ to the above quiver, we get the following quiver.
\[\xymatrixrowsep{5mm}
\xymatrixcolsep{5mm}
\xymatrix{ &{1\choose 1}\ar[rrr]& &&{1\choose 2}\ar[ld]|2 \ar@/_{1.5pc}/[lldd]|3&& {1\choose 3}\ar[ll]\ar@/_{-1pc}/[ldd]|3&
\\ 
&& & {2\choose 2}& && &{2\choose 3}\ar[llll]
\\
&&{3\choose 1}\ar[ru]|4\ar@/_{1pc}/[rrrruu]|3&& &{3\choose 2}\ar[lll]& &}
\]

By removing the source vertex ${1\choose 1}$ and
applying $\mu_{{1\choose 3}}\circ\mu_{{1\choose 2}}$ to the resulting quiver, we get the quiver associated with $S_{{\bf j}_{\geq 3}}^\emptyset(A)$,  where ${\bf j}_{\geq 3}=(j_3,\cdots,j_{10},j_{11})=(3,2,1,3,1,3,2,2,1).$ 

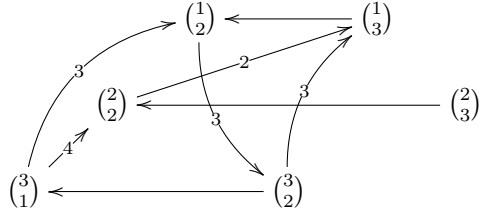
\[\xymatrixrowsep{5mm}
\xymatrixcolsep{5mm}
\xymatrix{ && &&{1\choose 2}\ar@/_{1pc}/[rdd]|3 && {1\choose 3}\ar[ll]&
\\ 
&& & {2\choose 2}\ar[rrru]|2& && &{2\choose 3}\ar[llll]
\\
&&{3\choose 1}\ar[ru]|4 \ar@/_{-1.5pc}/[rruu]|3&& &{3\choose 2}\ar[lll]\ar@/_{-1pc}/[ruu]|3& &}
\]

By applying $\mu_{{3\choose 2}}$ to the above quiver, we get the following quiver.

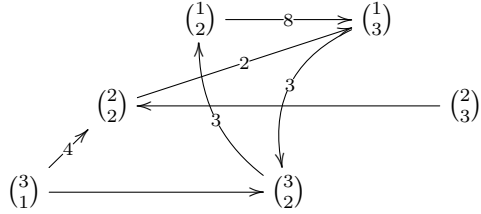
\[\xymatrixrowsep{5mm}
\xymatrixcolsep{5mm}
\xymatrix{ && &&{1\choose 2}\ar[rr]|8 && {1\choose 3}\ar@/_{1.5pc}/[ldd]|3&
\\ 
&& & {2\choose 2}\ar[rrru]|2& && &{2\choose 3}\ar[llll]
\\
&&{3\choose 1}\ar[rrr]\ar[ru]|4&& &{3\choose 2}\ar@/_{-1pc}/[luu]|3& &}
\]

By removing the source vertex ${3\choose 1}$ and applying $\mu_{{3\choose 2}}$ to the resulting quiver, we get the quiver
associated with $S_{{\bf j}_{\geq 4}}^\emptyset(A)$, where ${\bf j}_{\geq 4}=(j_4,\cdots,j_{10},j_{11})=(2,1,3,1,3,2,2,1).$ 
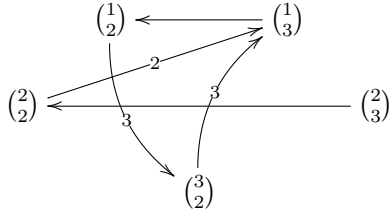
\[\xymatrixrowsep{5mm}
\xymatrixcolsep{5mm}
\xymatrix{ && &&{1\choose 2}\ar@/_{1pc}/[rdd]|3 && {1\choose 3}\ar[ll]&
\\ 
&& & {2\choose 2}\ar[rrru]|2& && &{2\choose 3}\ar[llll]
\\
&&&& &{3\choose 2}\ar@/_{-1pc}/[ruu]|3& &}
\]

By applying $\mu_{{2\choose 3}}$ to the above quiver, we get the following quiver.
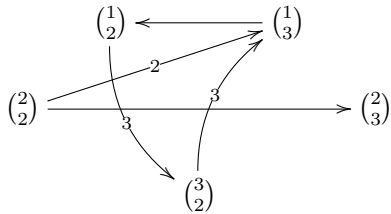
\[\xymatrixrowsep{5mm}
\xymatrixcolsep{5mm}
\xymatrix{ && &&{1\choose 2}\ar@/_{1pc}/[rdd]|3 && {1\choose 3}\ar[ll]&
\\ 
&& & {2\choose 2}\ar[rrru]|2\ar[rrrr]& && &{2\choose 3}
\\
&&&& &{3\choose 2}\ar@/_{-1pc}/[ruu]|3& &}
\]

By removing the source vertex ${2\choose 2}$ and applying  $\mu_{{2\choose 3}}$ to the resulting quiver, we get the quiver associated with $S_{{\bf j}_{\geq 5}}^\emptyset(A)$, where ${\bf j}_{\geq 5}=(j_5,\cdots,j_{10},j_{11})=(1,3,1,3,2,2,1).$ 

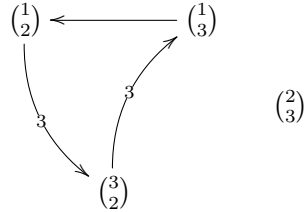
\[\xymatrixrowsep{5mm}
\xymatrixcolsep{5mm}
\xymatrix{ && &&{1\choose 2}\ar@/_{1pc}/[rdd]|3 && {1\choose 3}\ar[ll]&
\\ 
&& & & && &{2\choose 3}
\\
&&&& &{3\choose 2}\ar@/_{-1pc}/[ruu]|3& &}
\]

By applying $\mu_{{1\choose 3}}$ to the above quiver, we get the following quiver.

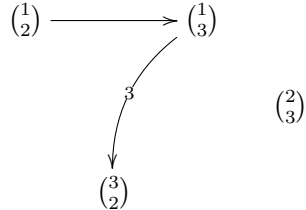
\[\xymatrixrowsep{5mm}
\xymatrixcolsep{5mm}
\xymatrix{ && &&{1\choose 2}\ar[rr] && {1\choose 3}\ar@/_{1pc}/[ldd]|3&
\\ 
&& & & && &{2\choose 3}
\\
&&&& &{3\choose 2}& &}
\]

We can continue the previous process or we just stop here, because the above quiver is acyclic and thus is already in the class $\mathcal P^\prime$.
\end{example}

\bibliographystyle{alpha}
\bibliography{myref}

\end{document}